\documentclass[10pt,reqno]{amsart}
\usepackage[left=2.5cm,right=2.2cm,top=2.5cm,bottom=3.5cm]{geometry}
\usepackage[colorlinks=true,allcolors=blue]{hyperref}
\usepackage{color}
\usepackage{amsmath, amssymb, amsthm}
\usepackage{mathtools}
\usepackage[noabbrev,capitalize,nameinlink,nosort]{cleveref}
\usepackage{enumerate}
\usepackage[hang,flushmargin]{footmisc}

\usepackage{booktabs}
\usepackage{makecell}
\usepackage{pgfplots}
\usepackage{pgfplotstable}
\pgfplotsset{compat=1.18}

\usepackage[shortlabels]{enumitem}
\crefformat{enumi}{#2#1#3}
\crefrangeformat{enumi}{#3#1#4 to~#5#2#6}
\crefmultiformat{enumi}{#2#1#3}
{ and~#2#1#3}{, #2#1#3}{ and~#2#1#3}

\let\originalleft\left
\let\originalright\right
\renewcommand{\left}{\mathopen{}\mathclose\bgroup\originalleft}
\renewcommand{\right}{\aftergroup\egroup\originalright}

\crefformat{equation}{#2(#1)#3}

\usepackage{aliascnt}
\newtheorem{lemma}{Lemma}[section]
\crefname{lemma}{Lemma}{Lemmas}
\Crefname{lemma}{Lemma}{Lemmas}

\newaliascnt{proposition}{lemma}
\newtheorem{proposition}[proposition]{Proposition}
\aliascntresetthe{proposition}
\crefname{proposition}{Proposition}{Propositions}
\Crefname{proposition}{Proposition}{Propositions}

\newaliascnt{theorem}{lemma}
\newtheorem{theorem}[theorem]{Theorem}
\aliascntresetthe{theorem}
\crefname{theorem}{Theorem}{Theorems}
\Crefname{theorem}{Theorem}{Theorems}

\newaliascnt{corollary}{lemma}

\aliascntresetthe{corollary}
\crefname{corollary}{Corollary}{Corollaries}
\Crefname{corollary}{Corollary}{Corollaries}

\newaliascnt{fact}{lemma}
\newtheorem{fact}[fact]{Fact}
\aliascntresetthe{fact}
\crefname{fact}{Fact}{Facts}
\Crefname{fact}{Fact}{Facts}

\newaliascnt{conjecture}{lemma}

\aliascntresetthe{conjecture}
\crefname{conjecture}{Conjecture}{Conjectures}
\Crefname{conjecture}{Conjecture}{Conjectures}

\newaliascnt{claim}{lemma}

\aliascntresetthe{claim}
\crefname{claim}{Claim}{Claims}
\Crefname{claim}{Claim}{Claims}

\theoremstyle{definition}

\newaliascnt{definition}{lemma}
\newtheorem{definition}[definition]{Definition}
\aliascntresetthe{definition}
\crefname{definition}{Definition}{Definitions}
\Crefname{definition}{Definition}{Definitions}

\newaliascnt{example}{lemma}

\aliascntresetthe{example}
\crefname{example}{Example}{Examples}
\Crefname{example}{Example}{Examples}

\theoremstyle{remark}

\newaliascnt{remark}{lemma}
\newtheorem{remark}[remark]{Remark}
\aliascntresetthe{remark}
\crefname{remark}{Remark}{Remarks}
\Crefname{remark}{Remark}{Remarks}

\newtheorem*{remark*}{Remark}

\newcommand*{\claimproofname}{Proof of claim}

\crefname{section}{Section}{Sections}
\Crefname{section}{Section}{Sections}
\crefname{subsection}{Section}{Sections}
\Crefname{subsection}{Section}{Sections}
\crefname{equation}{}{} 
\crefname{algorithm}{Algorithm}{Algorithms}
\Crefname{algorithm}{Algorithm}{Algorithms}

\AddToHook{cmd/appendix/after}{%
  \crefalias{section}{appendix}%
}

\newcommand{\mc}{\mathcal}

\newcommand{\mr}{\mathrm}

\newcommand{\eps}{\varepsilon}

\newcommand{\hide}[1]{}

\allowdisplaybreaks

\usepackage{parskip}
\makeatletter
\def\thm@space@setup{%
  \thm@preskip=\parskip \thm@postskip=0pt
}
\def\th@remark{%
  \thm@space@setup
  \thm@headfont{\itshape}%
  \normalfont
}
\makeatother

\let\le\leqslant
\let\ge\geqslant
\newcommand{\ZZ}{\mathbb{Z}}
\newcommand{\RR}{\mathbb{R}}
\newcommand{\NN}{\mathbb{N}}

\newcommand{\PP}{\mathbb{P}}

\newcommand{\EE}{\mathbb{E}}
\usepackage{algorithm}
\usepackage{algpseudocode}
\algtext*{EndIf}
\algtext*{EndFor}
\usepackage{bbm}

\newcommand{\cH}{\mathcal{H}}

\newcommand{\cD}{\mathcal{D}}

\newcommand{\Mod}[1]{\ (\text{mod}\ #1)}
\renewcommand{\Mod}[1]{{\ifmmode\text{\rm\ (mod~$#1$)}\else\discretionary{}{}{\hbox{ }}\rm(mod~$#1$)\fi}}
\newcommand{\defeq}{\vcentcolon=}

\DeclareMathOperator{\comp}{comp}
\DeclareMathOperator{\del}{del}

\DeclareMathOperator{\half}{half}
\DeclareMathOperator{\init}{init}
\DeclareMathOperator{\light}{light}
\DeclareMathOperator{\size}{size}
\DeclareMathOperator{\rep}{rep}
\DeclareMathOperator{\com}{com}
\DeclareMathOperator{\cir}{circ}
\newcommand{\dir}{\mathbf{d}}

\title[]{No-$(k+1)$-in-line problem for $k \ge 3$}

\author[]{Anubhab Ghosal}
\author[]{Ritesh Goenka}
\author[]{Alexandr Grebennikov}
\author[]{Peter Keevash}
\author[]{Matthew Kwan}
\author[]{Huy Tuan Pham}

\address{Mathematical Institute \\ University of Oxford}
\email{\{ghosal,goenka,keevash\}@maths.ox.ac.uk}
\address{Institute of Science and Technology Austria (ISTA)}
\email{\{aleksandr.grebennikov,matthew.kwan\}@ist.ac.at}
\address{Department of Mathematics, California Institute of Technology}
\email{htpham@caltech.edu}

\begin{document}

\begin{abstract}
  What is the maximum number of points one can place in an $n \times n$ grid such that every Euclidean line contains at most $k$ points? For $k = 2$, this is the notorious no-three-in-line problem of Dudeney. In this paper, we resolve this problem for all other $k$ (and sufficiently large $n$). Namely, for $k \ge 3$ and sufficiently large $n$, we show that this maximum is exactly $kn$.
  
  To prove this, our key observation is that in the regime $k\ge 3$, the problem is dominated in a certain statistical sense by the influence of a small number of ``heavy'' lines with many grid points. We apply a result of Ehard--Glock--Joos on pseudorandom hypergraph matchings to construct a set of size $kn - o(n)$ with at most $k$ points on each heavy line, and then a crude deletion argument yields a no-$(k+1)$-in-line set of nearly the same size. Finally, we use a randomised switching procedure to complete the construction (building upon ideas of Simkin and Luria).

  Using similar ideas, we also address the no-four-on-a-circle problem of Erd\H{o}s and Purdy. Namely, we prove the existence of a set of $2n - o(n)$ points in the $n \times n$ grid such that no four of these points lie on a circle or a line, improving on the previous construction of size $n - o(n)$ due to Dong and Xu. 
\end{abstract}

\maketitle

\section{Introduction}

The \emph{no-three-in-line} problem, posed by Dudeney in the early 20th century~\cite{dudeney-1917}, asks for the maximum number of points that can be placed on an $n \times n$ grid such that no three points are collinear. This problem is still open: the best known upper bound $2n$ comes from the observation that each of the $n$ horizontal lines can contain at most two points, while the best known lower bound $(1.5-o(1))n$ is given by the \emph{modular hyperbola} construction of Hall, Jackson, Sudbery, and Wild~\cite{HJSW-75} (improving on an earlier algebraic construction of Erd\H os; see \cite{roth-51}). For more history and background, we refer to the surveys~\cite{brass-moser-pach-05,eppstein-18}.

There are several different conjectures concerning the asymptotic behaviour of the answer to the no-three-in-line problem, including suggestions that it could be roughly $1.5n$~\cite{green-100-open-problems}, roughly $2n$~\cite{brass-moser-pach-05}, or somewhere in between~\cite{eppstein-18,guy-kelly-68}. For $n \le 60$, examples of no-three-in-line sets of size $2n$ were found by Prellberg~\cite{prellberg-26}.

A natural generalisation of this problem, first studied by Brass and Knauer~\cite{brass-knauer-03} in a more general context, is to ask for the maximum size of a subset of the $n \times n$ grid such that no $k+1$ points are collinear. Denoting this maximum by $f_k(n)$, the trivial upper bound is $f_k(n) \le kn$, since each of the $n$ rows (or columns) can have at most $k$ points. Also, note that this problem is only interesting for $n \ge k$: when $n \le k$, the whole grid has no $k+1$ collinear points, and thus $f_k(n) = n^2$.

Lefmann~\cite{lefmann-12} proved that $f_k(n) = \Omega(kn)$ for $n \ge k \ge 2$. Kov\'acs, Nagy, and Szab\'o~\cite{KNS-25} showed that $f_k(n) = kn$ as long as $n \ge k \ge C \sqrt{n \log n}$ for some absolute constant $C$, and Grebennikov and Kwan~\cite{grebennikov-kwan-25} recently extended this result to all $n \ge k \ge 10^{37}$.

For $3 \le k < 10^{37}$, the best known lower bounds are due to Kov\'acs, Nagy, and Szab\'o~\cite{KNS-25-algebraic}: by combining algebraic constructions with probabilistic ideas, they showed \cite[Theorem~1.5]{KNS-25-algebraic} that if $n$ is sufficiently large then $f_k(n) \ge (k-2-(k \bmod 2)) n$, and further improved the multiplicative constant for small values of $k$ (see \cite[Theorems 1.6 and 1.7]{KNS-25-algebraic}, e.g.\ they obtain that $f_3(n) \ge 1.973 n$).
In this paper, we prove that in fact $f_k(n) = kn$ for every $k \ge 3$ and sufficiently large $n$, thus showing that the trivial upper bound is tight.

\begin{theorem} \label{no-k+1-in-line-kn}
  Let $n, k$ be integers such that $k \ge 3$ and $n \ge \max(n_0, k)$ for some absolute constant $n_0$. Then there exists a set $S \subseteq [n]^2$ of size $kn$ such that every Euclidean line contains at most $k$ points of $S$.
\end{theorem}

Of course, the assumption $k\ge 3$ means that in this paper we do not say anything new about the no-three-in-line problem (which corresponds to the case $k=2$). In fact, there are some important statistical differences between the case $k=2$ and the case $k\ge 3$ (related to the relative significance of lines in different directions), which we discuss in \cref{rmk:k-illustration}.

\begin{remark}\label{rmk:algebraic}
  All previous constructions for small $k$ have been ``algebraic'', in the sense that they are based on algebraic curves in an affine plane $(\ZZ/p\ZZ)^2$ for some prime $p\le n$ (in fact, Green~\cite{green-100-open-problems} asked whether every ``large'' no-three-in-line set reduces to an algebraic curve modulo some prime). In contrast, our proof of \cref{no-k+1-in-line-kn} is non-algebraic and uses only fairly crude combinatorial properties of lines in $[n]^2$. 
\end{remark}

\begin{remark}\label{rmk:toroidal}
  In this paper we are concerned with \emph{Euclidean} lines in $[n]^2\subseteq \RR^2$. One can also ask about ``toroidal'' lines in $(\ZZ/n\ZZ)^2$, but this makes the problem substantially different. Indeed, the toroidal analogue of \cref{no-k+1-in-line-kn} is simply not true: if $n$ is an odd prime, the ``trivial bound for generalised arcs'' implies that every no-$(k+1)$-in-line set in $(\ZZ/n\ZZ)^2$ has size at most $(k-1)(n+1)+1$. See \cite{BH05} for a survey of this topic.
\end{remark}

\subsection{Higher dimensions}

A further generalisation of this problem, also introduced by Brass and Knauer~\cite{brass-knauer-03}, asks for the maximum size of a subset of the $d$-dimensional grid $[n]^d$ that contains at most $k$ points in each affine subspace of dimension $s$. Denoting this maximum by $f_{k, d, s}(n)$, the trivial upper bound is $f_{k, d, s}(n) \le k n^{d-s}$. Again, this problem is only interesting for $s+1 \le k \le n^s$: since any $s+1$ points lie in a common $s$-dimensional affine subspace, for $k \le s$ we have $f_{k, d, s}(n) = k$, and since the whole grid $[n]^d$ has at most $n^s$ points in each $s$-dimensional affine subspace, for $k \ge n^s$ we have $f_{k, d, s}(n) = n^d$.

Improving on earlier results of Brass--Knauer~\cite{brass-knauer-03} and Lefmann~\cite{lefmann-12}, it was observed by Sudakov and Tomon~\cite[Theorem~1.4]{sudakov-tomon-24} that $f_{k, d, s}(n) = \Omega_d(n^{d-s})$ when $k$ is sufficiently large in terms of $d$, using a connection to optimal subspace evasive sets over finite fields. Dvir and Lovett \cite{dvir-lovett} gave an algebraic construction of such objects; an alternative construction was observed by Conlon following the random algebraic method \cite{conlon}, and by Sudakov and Tomon~\cite{sudakov-tomon-24}. Ghosal, Goenka, and Keevash~\cite[Theorem~1.1]{ghosal-goenka-keevash-25} extended the bound $f_{k, d, s}(n) = \Omega_d(n^{d-s})$ to all $k \ge d+1$. Grebennikov and Kwan~\cite{grebennikov-kwan-25} proved that $f_{k, d, s}(n) = (1 + o(1)) k n^{d-s}$ when both $n$ and $k$ tend to infinity. In the current work, we extend this result to the case when $k$ is arbitrary satisfying $d + 1 \le k \le n^s$ and $n$ tends to infinity.

\begin{theorem} \label{higher-dimensions}
  Let $n, d, s, k$ be integers such that $d+1 \le k \le n^s$ and $1 \le s \le d-1$. Fix an arbitrary $\eta > 0$, and suppose that $n$ is sufficiently large in terms of $d$ and $\eta$. Then there exists a set $S \subseteq [n]^d$ of size at least $(1-\eta) k n^{d-s}$ such that each $s$-dimensional affine subspace contains at most $k$ points of $S$.
\end{theorem}

The assumption $k\ge d+1$ in \cref{higher-dimensions} has the same significance as the assumption $k\ge 3$ in \cref{no-k+1-in-line-kn}: the regime $s+1 \le k \le d$ has quite different statistical properties. In fact, it was observed in \cite{lefmann-12,suk-zeng-26} that for some values of $k, d, s$ in this regime, $f_{k, d, s}(n)$ is much smaller than $n^{d-s}$.

\subsection{No-four-on-a-circle problem}
A similar question, attributed to Erd\H{o}s and Purdy (see \cite[F3]{guy-81}), asks for the maximum number of points $f_{\cir}(n)$ that can be placed in the $n \times n$ grid so that no four points lie on the same circle or on the same line. Thiele~\cite{thiele-thesis,thiele-95} proved that $n/4 < f_{\cir}(n) \le 2.5 n - 1.5$, and Dong and Xu~\cite{dong-xu-25} recently used an algebraic construction to improve the lower bound to $n - o(n)$. Independently, Ghosal, Goenka, and Keevash~\cite[Corollary~1.4]{ghosal-goenka-keevash-25} showed that $f_{\cir}(n) \ge 7n / 12$ for large $n$ via a random deletion argument. We improve the lower bound further, demonstrating that $f_{\cir}(n) \ge 2n - o(n)$.

\begin{theorem} \label{no-four-on-a-circle-2n}
  Fix an arbitrary $\eta > 0$, and suppose that $n$ is sufficiently large in terms of $\eta$. Then there exists a set $S \subseteq [n]^2$ of size at least $(2-\eta) n$ that does not contain four points on a circle or on a line.
\end{theorem}

We do not believe that this bound is asymptotically sharp: in fact, using our methods (with some additional work), one should be able to show that $f_{\cir}(n) \ge (2+\lambda) n$ for some absolute constant $\lambda > 0$ and all large $n$ (see \cref{rmk:third-stage}). However, in the interest of keeping this paper short and simple, we do not pursue this here.

\begin{remark}
  To try to learn more about the problem, we ran some experiments with Google DeepMind's tool \emph{AlphaEvolve}~\cite{alphaevolve}, which uses an evolutionary search algorithm to produce point sets certifying lower bounds on $f_{\cir}(n)$. In \cref{sec:alphaevolve}, we present the results of these experiments, which weakly suggest that $f_{\cir}(n)/n$ might tend to a limit strictly between $2$ and $2.5$ as $n\to\infty$. We would like to thank Adam Zsolt Wagner for providing us with access to AlphaEvolve and personally assisting us in running many experiments. 
\end{remark}

\subsection{Proof ideas} \label{subsec:proof-ideas}
\cref{no-k+1-in-line-kn,higher-dimensions,no-four-on-a-circle-2n} are proved using similar methods, related to hypergraph matchings and random processes.
In this subsection, we focus on the proof of \cref{no-k+1-in-line-kn}, which is by far the most involved, and discuss aspects specific to \cref{higher-dimensions,no-four-on-a-circle-2n} where appropriate.

Our proof of \cref{no-k+1-in-line-kn} consists of two conceptual parts:
first, we construct a no-$(k+1)$-in-line set of size $kn - o(n)$ from a pseudorandom matching in a suitable hypergraph, and then we perform a sequence of $o(n)$ local modifications (switches) to obtain a no-$(k+1)$-in-line set of size exactly $kn$.
A similar approach was used by Simkin and Luria~\cite{luria-simkin-22} (and later in \cite{bowtell-keevash-21,simkin-23}) to obtain a lower bound on the number of sets $S \subseteq [n]^2$ of size $n$ containing at most one point in each row, column, and diagonal (also known as \emph{$n$-queens configurations}). While there are apparent similarities between the $n$-queens problem and the no-$(k+1)$-in-line problem, the actual implementation of this strategy in our setting requires several new ideas.

Since for $k \ge 10^{37}$ the desired statement is proved in \cite{grebennikov-kwan-25} (via a rather different approach), we focus on the case when $k$ is a fixed constant. Pick $\eps > 0$ sufficiently small in terms of $k$, and assume that $n$ is sufficiently large in terms of $k$ and $\eps$. 

\textbf{Approximate constructions via hypergraph matchings.} We begin by constructing a set $S_0 \subseteq [n]^2$ of size $kn - O_k(\eps n)$ that contains at most $k$ points on each ``$\eps$-heavy'' line, i.e.\ on every line with direction vector $(a, b) \in \ZZ^2$ such that $|a|, |b| \le 1/\eps$. Viewing each such line as a vertex and each point of $[n]^2$ as an edge incident to the lines that contain it, this reduces to finding an almost-perfect matching in (the $k$-blow-up of) a certain $O(1/\eps^2)$-uniform hypergraph. The existence of such a matching follows from the classical theorem of Pippenger and Spencer~\cite{pippenger-spencer-89} (proved using the celebrated \emph{R\"odl nibble} technique), and a result of Ehard, Glock, and Joos~\cite{ehard-glock-joos-20} further ensures that this matching can be taken to be pseudorandom (i.e., behaving similarly to an independent random subset of the edges of our hypergraph, in a suitable sense).

This pseudorandomness condition allows us to handle the remaining ``$\eps$-light'' lines via a crude deletion argument. Namely, it implies that the number of $(k+1)$-tuples of points of $S_0$ lying on the same $\eps$-light line is $O_k(\eps n)$, and hence deleting one point from each such tuple concludes the proof.

\begin{remark} \label{rmk:k-illustration}
  The final deletion step crucially relies on the assumption $k \ge 3$: under this assumption, it turns out that the problem is dominated by the few heaviest line directions. To give some intuition for this, suppose that $S_0$ is obtained not from a pseudorandom matching but by including each point of $[n]^2$ independently with probability $k/n$. Then the expected number of $(k+1)$-tuples in $S_0$ lying on the same $\eps$-light line is at most
  \[
  n^2 \cdot (k/n)^{k+1} \cdot \sum_{a, b} \left(\frac{n}{\max(|a|, |b|)} + 1\right)^k,  
  \]
  where the sum is over all possible slopes $a/b$ of $\eps$-light lines. It is not hard to check that this expression is $O_k(n \cdot \sum_{m = \lfloor 1/\eps \rfloor + 1}^{n} 1/m^{k-1})$. If $k \ge 3$, then the sum here is a tail of a convergent series: heuristically, this means that $\eps$-light lines contribute very little compared to $\eps$-heavy lines.
  On the other hand, if $k = 2$, then this sum is about $\log n$ (independently of $\eps$), and each ``dyadic scale'' of $m$ contributes to the sum equally.
\end{remark}

\begin{remark} \label{rmk:no-4-on-a-circle}
  \cref{higher-dimensions} is proved using an argument very similar to the one described above (in general, we need to consider primitive lattices in place of direction vectors). To prove \cref{no-four-on-a-circle-2n}, we first use an estimate by Huxley and Konyagin \cite{huxley-konyagin-09}, which implies that almost all cyclic quadrilaterals in the grid are isosceles trapezia. Then, we use a two-stage hypergraph matching argument. In the first stage, we use a pseudorandom matching in a suitable hypergraph to construct a point set $S_1$ of size $n - o(n)$ that contains at most one point on each $\eps$-heavy line. In the second stage, we construct a point set $S_2$ of size $n - o(n)$ that contains at most one point on each $\eps$-heavy line and such that $S_1 \cup S_2$ contains no isosceles trapezia with $\eps$-heavy parallel sides (again from a pseudorandom matching but in a different hypergraph depending on $S_1$). Finally, we delete one point from each of the remaining forbidden configurations in $S_1 \cup S_2$: isosceles trapezia with $\eps$-light parallel sides, cyclic quadrilaterals that are not isosceles trapezia, and quadruples of points on the same $\eps$-light line. 
\end{remark}

\textbf{Completion procedure.} Let $T = O_k(\eps n)$ be the ``size of the defect'': i.e., $kn$ minus the size of the approximate configuration $S_{\init}$ produced by the first part of the argument. Let $(c_1, \ldots, c_T)$ and $(r_1, \ldots, r_T)$ be the multisets of indices of columns and rows that remain ``unsaturated'' (i.e., the number of times a column index appears in the sequence $(c_1, \ldots, c_T)$ is $k-\#\{$points of $S_{\init}$ in this column$\}$, and similarly for row indices). Our goal is to perform $T$ ``switches'', as in the work of Simkin and Luria~\cite{luria-simkin-22}, such that the $t$-th switch increases the number of selected points in both the column and the row of $(c_t, r_t)$ by one while keeping the number of selected points in other rows and columns unchanged. 

Specifically, we say a point $(x, y)$ in the current set of selected points $S_{t-1}$ is an \emph{absorber} for some point $(c, r)$ if the set $S_t(x, y) \defeq S_{t-1} \setminus \{(x, y)\} \cup \{(x, r), (c, y)\}$ contains at most $k$ points on each line\footnote{This is a simplification: in the actual definition of absorbers (\cref{def:absorber}), we require a slightly stronger condition.}. Then, at step $t$, we pick a uniformly random absorber $(x, y)$ and set $S_t \defeq S_t(x, y)$. The key proposition (\cref{many-absorbers}) then states that, with high probability, at each step $t \le T$ we have $\Omega(n)$ absorbers to choose from. 

The proof of \cref{many-absorbers} combines the pseudorandomness of the initial configuration $S_{\init}$ and the randomness of the completion procedure. To make these two notions compatible, we need to extend the Ehard--Glock--Joos result slightly, so that it also allows $S_{\init}$ itself to be a ``spread'' random set, thus reusing some randomness of the first part of the proof. Using the pseudorandomness of $S_{\init}$, we show that it contains many absorbers\footnote{Strictly speaking, at this stage we only work with ``absorbers with respect to $\eps$-heavy lines''.} for every point $(c, r) \in [n]^2$. To bound the number of absorbers for $(c_t, r_t)$ that are ``lost'' during the completion procedure (up to step $t$), we use moment-based arguments. In turn, the relevant moment estimate follows from the ``spreadness'' of $S_{\init}$ and of an auxiliary random set $S_{\half}$ that includes a random one of the two points added at each of the previous switching steps. 

\begin{remark}
  To establish an exact bound in \cref{higher-dimensions} (analogous to \cref{no-k+1-in-line-kn}) using this approach, one would need a suitable version of the completion procedure for higher dimensions. Note that a subset of $[n]^d$ that contains exactly $k$ points in each axis-aligned affine subspace of dimension $s$ corresponds to a (multipartite) $(n, d, d-s, k)$-design, and one would need to transform a pseudorandom configuration of $(1-o(1))kn^{d-s}$ points into such a design via local modifications. 
  This might be related to the challenging problem of finding designs inside Erd\H{o}s--R\'enyi random hypergraphs of appropriate density (see \cite{DKP-26,jain-pham-24,KKKMO-23,keevash-22,SSS-23} for partial results in this direction).
\end{remark}

\subsection{Organisation of the paper} In \cref{sec:preliminaries}, we state our main technical tool: \cref{EGJ-spread} about pseudorandom hypergraph matchings (we formally deduce it from \cite{ehard-glock-joos-20} in \cref{sec:appendix}). The two parts of the proof of \cref{no-k+1-in-line-kn} (the approximate construction and the completion procedure) appear in \cref{sec:approximate,sec:completion}, respectively.
In \cref{sec:higher-dimensions}, we prove our higher-dimensional result (\cref{higher-dimensions}), and in \cref{sec:no-four-on-a-circle} we prove our bound for the no-four-on-a-circle problem (\cref{no-four-on-a-circle-2n}). 

\subsection*{Notation}

For a positive integer $n$, we write $[n]=\{1,\dots,n\}$. For a set $X$ and a positive integer $\ell$, we write $\binom{X}{\ell}$ for the collection of all subsets of $X$ of size $\ell$. We write $a \pm b$ to denote a quantity that differs from $a$ by at most $b$. We sometimes omit floor and ceiling symbols and assume large numbers are integers, when divisibility considerations are not important.
For a hypergraph $\cH$, we write $V(\cH)$ and $E(\cH)$ for its vertex set and edge set, respectively. Also, we write $\Delta(\cH)$ for its maximum degree and $\Delta_2(\cH)$ for its maximum codegree (i.e., the maximum number of edges containing a fixed pair of vertices).
For functions $f=f(n)$ and $g=g(n)$, we write $f=O(g)$ to mean that there is a constant $C$ such that $|f(n)|\le C|g(n)|$ for sufficiently large $n$. Similarly, we write $f=\Omega(g)$ to mean that there is a constant $c>0$ such that $f(n)\ge c|g(n)|$ for sufficiently large $n$. Finally, we write $f=o(g)$ to mean that $f(n)/g(n)\to0$ as $n\to\infty$. Subscripts on asymptotic notation indicate quantities that should be treated as constants. 

\subsection*{Acknowledgements}

The authors thank Rob Morris for helpful conversations. The first author is supported by the Clarendon Fund and Oxford Ryniker Lloyd Graduate Scholarship. The second author is supported by a joint Clarendon Fund and Exeter College SKP scholarship. The fourth author is supported by ERC Advanced Grant 883810. The third and the fifth authors are supported by ERC Starting Grant ``RANDSTRUCT'' No.~101076777. The sixth author is supported by a Clay Research Fellowship and NSF grant DMS-2543870.

\section{Preliminaries}
\label{sec:preliminaries}

The classical \emph{R\"odl nibble} technique can be used to prove that a regular hypergraph of degree $\Delta$ with small codegrees has a matching that covers a $1-o_{\Delta \to \infty}(1)$ fraction of its vertices (see \cite{pippenger-spencer-89}). Further refinements of this method \cite{alon-yuster-05,ehard-glock-joos-20,glock-joos-kim-kuhn-lichev-24} show that, in addition, one can demand that this matching is pseudorandom, in the sense that it shares certain statistical properties with the set obtained by including each edge independently with probability $1/\Delta$. \cref{EGJ-spread} below (essentially due to Ehard, Glock, and Joos~\cite{ehard-glock-joos-20}) is one such result that we will use heavily in this paper.  

\begin{definition}
  For a finite set $X$ and $\ell \in \NN$, an \emph{$\ell$-uniform test function\footnote{We use the terminology of \cite{glock-joos-kim-kuhn-lichev-24}; in \cite{ehard-glock-joos-20} these are called \emph{$\ell$-tuple weight functions}.} on~$X$} is a function $w: \binom{X}{\ell}\to \mathbb{R}_{\ge 0}$. For an arbitrary set $X' \subseteq X$, we write $w(X') \defeq \sum_{Y \in \binom{X'}{\ell}} w(Y)$.
  For $j \in [\ell]$ and $N \in \NN$, define 
  \[
  \Delta_j(w) \defeq \max_{|J| = j} \sum_{Y \supseteq J, |Y| = \ell} w(Y) \qquad \text{ and } \qquad B_N(w) \defeq \max_{j \in [\ell]} \Delta_j(w) N^j.
  \]
  Also, we say that a random subset $S$ of $X$ is \emph{$(q, C)$-spread} if for every set $Y \subseteq X$ of size at most $C$ we have $\PP[Y \subseteq S] \le q^{|Y|}$.

  For us, $X$ will always be the edge set $E(\cH)$ of some hypergraph. We say that an $\ell$-uniform test function $w$ on $E(\cH)$ is \emph{clean} if $w(E)=0$ whenever $E \in \binom{E(\mathcal{H})}{\ell}$ is not a matching. For an arbitrary test function $w$ on $E(\cH)$, let $\tilde{w}$ be its \emph{cleaning} defined by $\tilde{w}(E) = w(E)$ if $E$ is a matching and $\tilde{w}(E) = 0$ otherwise.
\end{definition}

\begin{lemma}
\label{EGJ-spread}
    Fix $\delta \in (0, 1)$ and $r, L \in \NN$ with $r \ge 2$. Let $\gamma \defeq \delta/(100L^2 r^2)$, and let $\Delta$ be sufficiently large in terms of $\delta, r, L$. Let $\cH$ be an $r$-uniform hypergraph with $\Delta(\cH) \le \Delta$, $\Delta_2(\cH) \le \Delta^{1-\delta}$, and $e(\cH) \le \exp(\Delta^{\gamma^2})$. Suppose that for each $\ell \in [L]$ we are given a set of clean $\ell$-uniform test functions $\mc{W}_{\ell}$ on $E(\cH)$ of size at most $\exp(\Delta^{\gamma^2})$. Then there exists a $((1 + \Delta^{-\gamma})/\Delta, \Delta^{\gamma})$-spread random matching $\mc{M}$ in $\cH$ that always satisfies
    \begin{equation}
    \label{eq:EGJ}
        w(\mc{M}) = \frac{(1 \pm \Delta^{-\gamma})w(E(\cH)) \pm 2B_{\Delta}(w) \Delta^{\delta}}{\Delta^{\ell}}
    \end{equation}
    for each $\ell \in [L]$ and $w \in \mc{W}_{\ell}$.
\end{lemma}

This is a slight extension of \cite[Theorem 1.3]{ehard-glock-joos-20} which follows from the same construction, so we defer the proof to \cref{sec:appendix}.
The main new feature of \cref{EGJ-spread} (which is implicit in \cite{ehard-glock-joos-20} but made explicit here) is the spreadness of the resulting random matching $\mc{M}$. It has the following useful consequence: for an \emph{arbitrary} additional $\ell$-uniform test function $w_0$ with $\ell \le \Delta^{\gamma}$ (possibly, such that $B_{\Delta}(w_0)$ is large compared to $w_0(E(\cH))$), we have 
\[
\EE[w_0(\mc{M})] \le \left(\frac{1 + \Delta^{-\gamma}}{\Delta}\right)^{\ell} w_0(E(\cH)),
\]
and hence by Markov's inequality $w_0(\mc{M}) = O(w_0(E(\cH)) / \Delta^{\ell})$ with probability at least $1/2$.

We also record a simple but extremely convenient fact about spread random sets.

\begin{fact} \label{spread-union}
  Let $S_1,S_2$ be coupled random subsets of $X$, such that:
  \begin{itemize}
      \item $S_1$ is $(q_1, C_1)$-spread, and
      \item if we condition on any outcome of $S_1$, then the conditional distribution of $S_2$ is $(q_2, C_2)$-spread.
  \end{itemize}
  Then $S_1 \cup S_2$ is $(q_1 + q_2, \min(C_1, C_2))$-spread.
\end{fact}
\begin{proof}
  Consider a set $Y \subseteq X$ of size at most $\min(C_1, C_2)$. Then,   
  \begin{align*}
  \PP[Y \subseteq S_1 \cup S_2] \le \sum_{Y_1 \sqcup Y_2 = Y} \EE[\mathbf{1}_{Y_1 \subseteq S_1} \cdot \PP[Y_2 \subseteq S_2 \mid S_1]] &\le \sum_{Y_1 \sqcup Y_2 = Y} \PP[Y_1 \subseteq S_1] \cdot q_2^{|Y_2|} \\
  &\le \sum_{Y_1 \sqcup Y_2 = Y} q_1^{|Y_1|} q_2^{|Y_2|} = (q_1 + q_2)^{|Y|}. \qedhere
  \end{align*}
\end{proof}

\section{Approximate construction}
\label{sec:approximate}

In this section we give a short proof of the following approximate version of \cref{no-k+1-in-line-kn}. This constitutes the first part of the proof of \cref{no-k+1-in-line-kn}, and also serves as an illustration of our approach.

\begin{proposition} \label{no-k+1-in-line-approximate}
  Fix an arbitrary $\eta > 0$, and let $n, k$ be integers such that $k \ge 3$ and $n$ is sufficiently large in terms of $k$ and $\eta$. Then there exists a set $S \subseteq [n]^2$ of size at least $(1-\eta)kn$ such that every Euclidean line contains at most $k$ points of $S$. 
\end{proposition}

Let $\mathcal{D}$ denote the set of possible directions of lines that intersect $[n]^2$ in at least two points:
\[
\mc{D} \defeq \{(a, b) \in \ZZ^2 : |a| < n,\; |b| < n, \; \gcd(a, b) = 1, \; a > 0 \text{ or } (a = 0 \text{ and } b > 0)\}.
\]
For $\eps \in (0, 1]$, we say that a direction $\mathbf{d} \in \mc{D}$ is \emph{$\eps$-heavy} if $1/\|\mathbf{d}\|_{\infty} \ge \eps$; otherwise, we say that it is \emph{$\eps$-light}. Let  $\mc{D}_\eps \subseteq \mc{D}$ denote the set of $\eps$-heavy directions, and note that $|\mc{D}_\eps| = O(1/\eps^2)$. We say that a Euclidean line is $\eps$-heavy (resp. $\eps$-light) if its direction is $\eps$-heavy (resp. $\eps$-light). To handle the light lines via a deletion argument, we will use the following counting lemma.

\begin{lemma} \label{light-lines-triples}
  For every $\eps \in (0, 1]$ and $p \in [n]^2$, there are $O(\eps n^3)$ triples of points $\{p_1, p_2, p_3\} \subseteq [n]^2 \setminus \{p\}$ such that $p, p_1, p_2, p_3$ lie on the same $\eps$-light line.
\end{lemma}
\begin{proof}
  For each $m \in \NN$, there are at most $4m$ directions $\dir \in \mc{D}$ with $\|\dir\|_\infty = m$. For each $\eps$-light direction $\dir$, the line through $p$ in direction $\dir$ contains at most $n/\|\dir\|_\infty$ other points of $[n]^2$, and hence there are at most $(n/\|\dir\|_\infty)^3$ choices for the triple $\{p_1, p_2, p_3\}$ on this line. Taking the sum over all $\eps$-light directions, we conclude that the total number of such triples is at most
  \[
  n^3 \sum_{\dir \in \mc{D} \setminus \mc{D}_\eps} \frac{1}{\|\dir\|_\infty^3} \le n^3 \sum_{m = \lfloor 1/\eps \rfloor + 1}^n \frac{4m}{m^3} = O(\eps n^3). \qedhere
  \]
\end{proof}

\begin{definition} \label{def:hypergraph-approximate}
  Let $\cH_\eps$ be the following $|\mathcal{D}_\eps|$-uniform $|\mathcal{D}_\eps|$-partite hypergraph: the vertices of $\cH_\eps$ are the $\eps$-heavy lines which intersect $[n]^2$, and for each $p \in [n]^2$ we put an edge $\{L \in V(\cH_\eps) : p \in L \}$.
  Then, let $\cH^{(k)}_{\eps}$ be the union of $k$ disjoint copies of $\cH_{\eps}$.
  We identify the edge set of $\cH^{(k)}_{\eps}$ with the Cartesian product $[n]^2 \times [k]$.
\end{definition}

\begin{definition} \label{def:test-functions-approximate}
  Define the $\{0,1\}$-valued test functions $w_{\size}, w_{\rep}$, and $w_{\light}$ on $E(\cH^{(k)}_\eps)$ as follows:
  \begin{itemize}
    \item $w_{\size}$ is $1$-uniform, and $w_{\size}(\{(p, i)\}) = 1$ for every $(p, i) \in [n]^2 \times [k]$;
    \item $w_{\rep}$ is $2$-uniform, and $w_{\rep}(\{(p_1, i_1), (p_2, i_2)\}) = 1$ if and only if $p_1 = p_2$;
    \item $w_{\light}$ is $4$-uniform, and $w_{\light}(\{(p_1, i_1), (p_2, i_2), (p_3, i_3), (p_4, i_4)\}) = 1$ if and only if $p_1, p_2, p_3, p_4$ are distinct and lie on the same $\eps$-light line.
  \end{itemize}
  Note that all these test functions are clean.
\end{definition}

\begin{lemma} \label{test-functions-approximate-properties}
  The test functions $w_{\size}$, $w_{\rep}$, and $w_{\light}$ defined above satisfy the following properties:
  \begin{itemize}
    \item $w_{\size}([n]^2 \times [k]) = kn^2$ and $B_n(w_{\size}) = n$;
    \item $w_{\rep}([n]^2 \times [k]) = \binom{k}{2}n^2$ and $B_n(w_{\rep}) = n^2$;
    \item $w_{\light}([n]^2 \times [k]) = O_k(\eps n^5)$ and $B_n(w_{\light}) = O_k(n^4)$.
  \end{itemize}
\end{lemma}
\begin{proof}
  The first two items are immediate from the definitions of $w_{\size}$ and $w_{\rep}$. For the third item, note that 
  \[
  w_{\light}([n]^2 \times [k]) = k^4 \cdot \Big|\Big\{\{p_0, p_1, p_2, p_3\} \in \binom{[n]^2}{4} : p_0, p_1, p_2, p_3 \text{ lie on the same $\eps$-light line}\Big\}\Big|.
  \]
  Summing the bound given by \cref{light-lines-triples} over all $p_0 \in [n]^2$, we conclude that $w_{\light}([n]^2 \times [k]) = O_k(\eps n^5)$. Similarly, again by \cref{light-lines-triples}, 
  \[
  \Delta_1(w_{\light}) = k^3 \cdot \max_{p \in [n]^2} \Big|\Big\{\{p_1, p_2, p_3\} \in \binom{[n]^2 \setminus \{p\}}{3} : p, p_1, p_2, p_3 \text{ lie on the same $\eps$-light line}\Big\}\Big| = O_k(\eps n^3).
  \]
  Finally, since two points determine a line, for $2 \le j \le 4$ we have $\Delta_j(w_{\light}) \le (kn)^{4-j}$. Therefore, $B_n(w_{\light}) = \max\limits_{j \in [4]} \Delta_j(w_{\light}) n^j = O_k(n^4)$.
\end{proof}

\begin{proof}[\textbf{Proof of \cref{no-k+1-in-line-approximate}}]
    Set $\eps \defeq c\eta$ for a sufficiently small $c = c(k) > 0$. Throughout the proof, we assume that $n$ is sufficiently large in terms of $k$ and $\eps$.

    We would like to apply \cref{EGJ-spread} to the hypergraph $\cH^{(k)}_\eps$ from \cref{def:hypergraph-approximate} with $\Delta = n$ and $r = |\mathcal{D}_\eps| = O(1/\eps^2)$ and $\delta = 0.1$ and $L = 4$, and test functions $w_{\size}$, $w_{\rep}$, and $w_{\light}$ from \cref{def:test-functions-approximate} (in fact, here we do not need the spreadness guarantee; we only use the existence of a matching satisfying \cref{eq:EGJ}). To check the required hypotheses, we first note that $\Delta(\cH^{(k)}_{\eps}) = n$ and, since a pair of lines share at most one point, $\Delta_2(\cH^{(k)}_{\eps}) = 1 \le n^{1-\delta}$. Since $\gamma = \delta/(100L^2r^2) = \Theta_{k, \eps}(1)$, we have $|E(\cH^{(k)}_\eps)| = kn^2 \le \exp(n^{\gamma^2})$, and the number of test functions is $3 \le \exp(n^{\gamma^2})$. So, \cref{EGJ-spread} gives us a matching $\mc{M}$ in $\cH^{(k)}_\eps$ such that 
    \[
    |\mc{M}| = w_{\size}(\mc{M}) = (1 \pm n^{-\gamma}) kn \pm 2n^{\delta}, \qquad w_{\rep}(\mc{M}) = O_k(n^{\delta}), \qquad w_{\light}(\mc{M}) = O_k(\eps n). 
    \]
    Viewing $\mc{M}$ as a subset of $[n]^2 \times [k]$, let $S_{\mc{M}}$ be its projection onto $[n]^2$. Then, 
    \[
    |S_{\mc{M}}| \ge |\mc{M}| - w_{\rep}(\mc{M}) = kn - O_k(n^{1 - \gamma}).
    \]
    Since $\mc{M}$ is a matching in $\cH^{(k)}_{\eps}$, $S_{\mc{M}}$ contains at most $k$ points on each $\eps$-heavy line. Let $S_{\del} \subseteq S_{\mc{M}}$ be the set obtained by including one point from each quadruple of points in $S_{\mc{M}}$ that lie on the same $\eps$-light line. Then, $|S_{\del}| \le w_{\light}(\mc{M}) = O_k(\eps n)$, and hence 
    \[
    |S_{\mc{M}} \setminus S_{\del}| = kn - O_k(n^{1-\gamma}) - O_k(\eps n) \ge (1-\eta)kn.
    \] 
    Since $S_{\mc{M}} \setminus S_{\del}$ contains at most $3 \le k$ points on each $\eps$-light line by construction, this completes the proof.
\end{proof}

\section{Completion procedure}
\label{sec:completion}

In this section, our goal is to prove \cref{no-k+1-in-line-kn}. Since for $n \ge k \ge 10^{37}$ the desired result is just \cite[Theorem 1.1]{grebennikov-kwan-25}, we focus on the case when $k$ is fixed and $n$ is sufficiently large in terms of $k$.
The dependence of $n$ on $k$ will be moderated via the intermediate parameters $\rho, \alpha, \eps \in (0, 1)$ used throughout this section that satisfy
\begin{equation} \label{eq:dependencies-of-parameters}
1/k \gg \rho \gg \alpha \gg \eps \gg 1/n,
\end{equation}
where $a \gg b$ means that we take $b$ to be sufficiently small in terms of $a$. 

We briefly explain the role of each intermediate parameter. The role of $\eps$ is essentially the same as in \cref{sec:approximate}: it is a threshold that separates the heavy directions that correspond to the parts of $\cH^{(k)}_\eps$ from the light directions that are handled via a deletion argument. The parameter $\alpha$ is a similar ``heaviness threshold'', which will be used for the completion procedure after the deletion argument (for technical reasons we need $\alpha$ to be much larger than $\eps$).
Finally, $\rho$ controls the randomness of the completion procedure: we will show that, with high probability, at each step we have at least $\rho n$ available absorbers to choose from.

We say that a direction $\dir \in \mc{D}$ is \emph{non-trivial} if $\dir \notin \{(1,0), (0,1)\}$; similarly, a line is \emph{non-trivial} if it is neither vertical nor horizontal. Let $\mc{D}'_{\eps} \defeq \mc{D}_\eps \setminus \{(1,0), (0,1)\}$ be the set of non-trivial $\eps$-heavy directions. Also, we say that a direction $\dir$ (or a line with such direction) is \emph{irrelevant} if $\|\dir\|_{\infty} \ge n/3$. Note that an irrelevant line contains at most $3$ points of the grid $[n]^2$, and thus is indeed not relevant for our problem.

\begin{definition} \label{def:absorber}
A point $(x, y) \in S \subseteq [n]^2$ is called an \emph{absorber} in $S$ for a point $(c, r) \in [n]^2$ if 
\begin{enumerate}[{\bfseries{A\arabic{enumi}}}]
  \item\label{A1} $(x, r) \notin S$, $(c, y) \notin S$; 
  \item\label{A2} $(x, y) \neq (c, r)$, and the line through $(x, r)$ and $(c, y)$ is irrelevant;
  \item\label{A3} every non-trivial line through $(x, r)$ or $(c, y)$ contains at most $k-1$ points of $S$.
\end{enumerate}
Let $\mc{A}_S(p)$ denote the set of absorbers in $S$ for $p$.
\end{definition}

To illustrate the usefulness of absorbers, consider a set $S \subseteq [n]^2$ with at most $k$ points on each line and a point $(c, r)$ such that its row and column contain at most $k-1$ points of $S$. If $(x, y) \in S$ is an absorber for $(c, r)$, then the set $S' \defeq (S \setminus \{(x, y)\}) \cup \{(x, r), (c, y)\}$ also contains at most $k$ points on each line, and has size $|S| + 1$.

For future use, we record the observation that a positive fraction of the grid points satisfy condition \cref{A2}. For a point $(c, r) \in [n]^2$, define 
\[
I(c, r) \defeq \{(x, y) \in [n]^2 \setminus \{(c, r)\} : \text{the line through $(x, r)$ and $(c, y)$ is irrelevant}\}.
\]

\begin{lemma} \label{irrelevant-lines}
  $|I(c, r)| \ge 0.05 n^2$ for each $(c, r) \in [n]^2$.
\end{lemma}
\begin{proof}
  By reflecting the grid if necessary, we may assume that $c \le n/2$ and $r \le n/2$. By definition, if 
  \[
  \max(|x-c|, |y-r|) \ge n/3 \quad \text{ and } \quad \gcd(x-c, y-r) = 1,
  \]
  then the line through $(x, r)$ and $(c, y)$ is irrelevant. Hence, for every integer vector $(a, b) \in [0, n/2] \times [n/3, n/2]$ with $\gcd(a, b) = 1$, the point $(c+a, r+b)$ belongs to $I(c, r)$. A classical result (likely dating back to Minkowski, see e.g.~\cite{BMNR-20}) states that for a convex polygon $Q$, the number of integer points with coprime coordinates inside $m Q$ is equal to $\frac{6}{\pi^2} \cdot \mathrm{area}(Q) \cdot m^2 + o(m^2)$ as $m \to \infty$. Applying this result to the rectangle $[0, n/2] \times [n/3, n/2]$, we conclude that 
  \[
  |I(c, r)| \ge \frac{6}{\pi^2} \cdot \frac{n^2}{12} + o(n^2) > 0.05 n^2. \qedhere
  \]
\end{proof}

As in \cref{sec:approximate}, our construction involves a pseudorandom matching in the hypergraph $\cH^{(k)}_\eps$ from \cref{def:hypergraph-approximate}. To show that the resulting configuration can be completed to a no-$(k+1)$-in-line set of size $kn$, we need to ensure that at each step of the completion procedure we have enough available absorbers. In turn, to find these absorbers we need finer control over the pseudorandomness of our matching, which is achieved via additional test functions for heavy lines and via spreadness for light lines.

\subsection{New test functions} For a point $p \in [n]^2$ and a direction $\dir \in \mc{D}$, let $L_{\dir}(p)$ denote the intersection of $[n]^2$ with the line through $p$ in direction $\dir$, and let $L^*_{\dir}(p) := L_{\dir}(p) \setminus \{p\}$.

\begin{definition} \label{def:omega-com}
  Fix a point $(c, r) \in [n]^2$. For a point $(x, y) \in [n]^2$, we say that a point $p' \in [n]^2$ is \emph{$\eps$-common} between $(x, r)$ and $(c, y)$ if it belongs to $L^*_{\dir_1}(x, r) \cap L^*_{\dir_2}(c, y)$ for some non-trivial $\eps$-heavy directions $\dir_1, \dir_2 \in \mc{D}'_\eps$. 
  Let $w^{(c, r)}_{\com}$ be the $2$-uniform test function on $E(\cH^{(k)}_\eps) = [n]^2 \times [k]$ defined as follows. For a pair of edges $E\subseteq [n]^2\times [k]$, let $w^{(c, r)}_{\com}(E)$ be the number of choices of $(x, y) \in I(c, r)$ and $p' \in [n]^2$ and $i' \in [k]$, such that $p'$ is $\eps$-common between $(x, r)$ and $(c, y)$, and such that $E = \{((x, y), 1), (p', i')\}.$
  (Note that $w^{(c, r)}_{\com}$ takes values in $\{0, 1, 2\}$.)
\end{definition}

\begin{lemma} \label{omega-com-properties}
  For every point $(c, r) \in [n]^2$, we have 
  \[
  w_{\com}^{(c, r)}([n]^2 \times [k]) = O_{k, \eps}(n^2), \qquad B_n(w_{\com}^{(c, r)}) = O_{k, \eps}(n^2).
  \]
\end{lemma}
\begin{proof}
  For each $(x, y) \in I(c, r)$ the line through $(x, r)$ and $(c, y)$ is irrelevant and thus surely not $\eps$-heavy. Hence, since two different lines meet in at most one point and $\cD_\eps' \subseteq \{(a, b) \in \ZZ^2 : |a|, |b| \le 1/\eps\}$, the number of $\eps$-common points between $(x, r)$ and $(c, y)$ is at most $|\cD_\eps'|^2 = O(1/\eps^4)$. Similarly, for each point $p' \in [n]^2$, $\eps$-heavy lines through $p'$ intersect the column and row of $(c, r)$ in at most $|\cD_\eps'|$ points, which leaves at most $|\cD_\eps'|^2 = O(1/\eps^4)$ options for the point $(x, y)$ such that $p'$ is $\eps$-common between $(x, r)$ and $(c, y)$. This implies that $\Delta_1(w_{\com}^{(c, r)}) = O_{k, \eps}(1)$. Also recall that $\Delta_2(w_{\com}^{(c, r)})\le 2$, so
  \begin{align*}
  &w_{\com}^{(c, r)}([n]^2 \times [k]) \le n^2 k \Delta_1(w_{\com}^{(c, r)}) = O_{k, \eps}(n^2), \\
  &B_n(w_{\com}^{(c, r)}) = \max(\Delta_1(w_{\com}^{(c, r)}) n, \Delta_2(w_{\com}^{(c, r)}) n^2) = O_{k, \eps}(n^2). \qedhere
  \end{align*}
\end{proof}

The next test function encodes the ``obstructions'' that preclude a point $(x, y)$ from being an absorber for $(c, r)$.

\begin{definition} \label{def:omega-D1-D2}
  For sets of non-trivial $\eps$-heavy directions $D_1, D_2 \subseteq \mathcal{D}_{\eps}'$ and a point $(c, r) \in [n]^2$, let $w_{D_1, D_2}^{(c, r)}$ be the $(k(|D_1|+|D_2|)+1)$-uniform test function on $E(\cH^{(k)}_\eps) = [n]^2 \times [k]$ defined as follows. For a set $E\subseteq [n]^2\times [k]$ of size $k(|D_1|+|D_2|)+1$, let $w_{D_1, D_2}^{(c, r)}(E)$ be the number of choices of $(x,y)\in I(c, r)$ and of $p_{\dir_1,i}\in [n]^2$ and $p_{\dir_2,i}\in [n]^2$ for all $\dir_1\in D_1$ and $\dir_2\in D_2$ and $i\in[k]$, such that
  \begin{enumerate}[{\bfseries{P\arabic{enumi}}}]
    \item\label{P1} $p_{\dir_1,i} \in L^*_{\dir_1}(x, r)$ for each $\dir_1 \in D_1$ and $i \in [k]$, and $p_{\dir_2,i} \in L^*_{\dir_2}(c, y)$ for each $\dir_2 \in D_2$ and $i \in [k]$,
    \item\label{P2} the points $p_{\dir_1,i}$ and $p_{\dir_2,i}$ are pairwise distinct and \emph{not} $\eps$-common between $(x, r)$ and $(c, y)$,
  \end{enumerate}
  and such that
  \[
   E = \{((x, y), 1)\} \cup \{(p_{\dir_1,i}, i)\}_{\dir_1 \in D_1, i \in [k]} \cup \{(p_{\dir_2,i}, i)\}_{\dir_2 \in D_2, i \in [k]}.
  \]
\end{definition}

To motivate this definition, for a matching $\mc{M}$ in $\cH^{(k)}_\eps$ and a point $(c, r) \in [n]^2$, consider the expression
\begin{equation} \label{eq:inclusion-exclusion}
\sum_{D_1, D_2 \subseteq \mc{D}'_{\eps}} (-1)^{|D_1|+|D_2|}w_{D_1, D_2}^{(c, r)}(\mc{M}).
\end{equation}
By inclusion-exclusion, it counts (up to lower-order terms) the number of points $(x, y) \in I(c, r)$ such that $((x, y), 1) \in \mc{M}$ and the projection of $\mc{M} \subseteq [n]^2 \times [k]$ onto $[n]^2$ contains at most $k-1$ points on each non-trivial $\eps$-heavy line through $(x, r)$ or $(c, y)$. To estimate the expression \cref{eq:inclusion-exclusion}, we will combine \cref{EGJ-spread} with the following statement.

\begin{lemma} \label{omega-D1-D2-properties}
  For every pair of sets $D_1, D_2 \subseteq \mathcal{D}_{\eps}'$ and a point $(c, r) \in [n]^2$, we have
  \begin{align*}
  \tilde{w}_{D_1, D_2}^{(c, r)}([n]^2 \times [k]) &= \sum_{(x, y)\in I(c, r)}\prod_{\dir_1 \in D_1}|L^*_{\dir_1}(x, r)|^k\prod_{\dir_2 \in D_2} |L^*_{\dir_2}(c, y)|^k + O_{k,\eps}(n^{k(|D_1|+|D_2|)+1}),\\    
  B_n(\tilde{w}_{D_1, D_2}^{(c, r)}) &\le B_n(w_{D_1, D_2}^{(c, r)}) = O_{k,\eps}(n^{k(|D_1|+|D_2|)+1}),
  \end{align*}
  where $\tilde{w}_{D_1, D_2}^{(c, r)} \defeq \mathbf{1}_{\{\text{$E$ : $E$ is a matching in $\cH^{(k)}_\eps$}\}} \cdot w_{D_1, D_2}^{(c, r)}$ is the cleaning of $w_{D_1, D_2}^{(c, r)}$.
\end{lemma}
\begin{proof}
  For brevity, denote $K \defeq k(|D_1|+|D_2|) + 1 = O_{k, \eps}(1)$. First, we check that for every $j \in [K]$ we have $\Delta_j(w_{D_1, D_2}^{(c, r)}) = O_{k, \eps}(n^{K-j})$, which would imply the desired bound on $B_n(w_{D_1, D_2}^{(c, r)})$.
  Consider a set $J \subseteq [n]^2 \times [k]$ of size $j \in [K]$. We need to bound the number of sequences 
  \begin{equation} \label{eq:sequence-E}
    \vec{E} = \Big(((x, y), 1), ((p_{\dir_1, i}, i))_{\dir_1 \in D_1, i \in [k]}, ((p_{\dir_2, i}, i))_{\dir_2 \in D_2, i \in [k]} \Big)
  \end{equation}
  of elements of $[n]^2 \times [k]$ satisfying \cref{P1,P2} from \cref{def:omega-D1-D2} whose underlying set $E$ contains $J$.
  Fix one of the $O_{k, \eps}(1)$ ways to assign distinct ``positions'' in this sequence to the elements of $J$. If some element of $J$ plays the role of $((x, y), 1)$, then by \cref{P1} the number of ways to fill each of the $K-j$ unoccupied positions is at most $n$, which gives a total of at most $n^{K-j}$ such sequences $\vec{E}$. Otherwise, consider an arbitrary $(p, i) \in J$ which plays the role of $(p_{\dir_1, i}, i)$ for some $\dir_1 \in D_1$ or $(p_{\dir_2, i}, i)$ for some $\dir_2 \in D_2$. In the first case $(x, y)$ must share a column with the intersection point of $L_{\dir_1}(p)$ and the row of $(c, r)$, and in the second case $(x, y)$ must share a row with the intersection point of $L_{\dir_2}(p)$ and the column of $(c, r)$. Either way, this leaves at most $n$ options for $(x, y)$, and hence at most $n \cdot n^{K-j-1} = n^{K-j}$ options for the sequence $\vec{E}$. 

  Next, we prove the asymptotic formula for $w_{D_1, D_2}^{(c, r)}([n]^2 \times [k])$, which is again just the number of sequences $\vec{E}$ as in \cref{eq:sequence-E} that satisfy \cref{P1,P2}. Fix a point $(x, y) \in I(c, r)$, and note that there are at most $|\mc{D}'_\eps|^2 = O(1/\eps^4)$ $\eps$-common points between $(x, r)$ and $(c, y)$ (because the line through $(x, r)$ and $(c, y)$ is not $\eps$-heavy). Picking the points $(p_{\dir_1, i})_{\dir_1 \in D_1, i \in [k]}$ and $(p_{\dir_2, i})_{\dir_2 \in D_2, i \in [k]}$ one by one, we have $|L^*_{\dir_1}(x, r)| - O_{k, \eps}(1)$ choices for each $p_{\dir_1, i}$ and $|L^*_{\dir_2}(c, y)| - O_{k, \eps}(1)$ choices for each $p_{\dir_2, i}$ (because \cref{P2} requires these points to be distinct and not $\eps$-common between $(x, r)$ and $(c, y)$). As a consequence, 
  \begin{equation} \label{eq:omega-D1-D2-no-cleaning}
  w^{(c, r)}_{D_1, D_2}([n]^2 \times [k]) = \sum_{(x, y) \in I(c, r)} \prod_{\dir_1 \in D_1} \left(|L^*_{\dir_1}(x, r)| - O_{k, \eps}(1)\right)^k \prod_{\dir_2 \in D_2} \left(|L^*_{\dir_2}(c, y)| - O_{k, \eps}(1)\right)^k.
  \end{equation}
  It remains to take into account the cleaning. We will show that the number of sequences $\vec{E}$ such that $\tilde{w}(E) \neq w(E)$ is $O_{k, \eps}(n^K)$. Together with \cref{eq:omega-D1-D2-no-cleaning}, this would yield the desired asymptotic formula for $\tilde{w}_{D_1, D_2}^{(c, r)}([n]^2 \times [k])$. By the definition of cleaning, the underlying set $E$ of each such sequence contains two elements $(p'_1, i)$ and $(p'_2, i)$ such that the points $p'_1$ and $p'_2$ lie on the same $\eps$-heavy line in some direction $\dir_0$. 
  
  Consider the case when neither of $p'_1$ and $p'_2$ plays the role of $(x, y)$ in $\vec{E}$. In this case, $p'_1 \in L^*_{\dir_1}(p_1)$ and $p'_2 \in L^*_{\dir_2}(p_2)$ for some non-trivial $\eps$-heavy directions $\dir_1$, $\dir_2$ and points $p_1, p_2 \in \{(x, r), (c, y)\}$ such that $(\dir_1, p_1) \neq (\dir_2, p_2)$. Since $p'_1$ and $p'_2$ are not $\eps$-common between $(x, r)$ and $(c, y)$ by \cref{P2}, this implies that $\dir_0 \neq \dir_1$ and $\dir_0 \neq \dir_2$. Thus, fixing one of $O_{k, \eps}(n^3)$ choices of $(x, y)$, $\dir_0$, $\dir_1$, $\dir_2$, $p_1$, $p_2$, and $p'_1 \in L^*_{\dir_1}(p_1)$ determines $p'_2$. Each of the $K-3$ remaining positions in $\vec{E}$ can be filled in at most $n$ ways, and hence the number of such sequences $\vec{E}$ is $O_{k, \eps}(n^3) \cdot n^{K-3} = O_{k, \eps}(n^K)$.
  
  Similarly, consider the case when (say) $p'_1 = (x, y)$ and $p'_2 \in L^*_{\dir}(p)$ for some non-trivial $\eps$-heavy direction $\dir$ and a point $p \in \{(x, r), (c, y)\}$. Then, clearly, $\dir_0 \neq \dir$, and thus fixing one of $O_{k, \eps}(n^2)$ choices of $p'_1 = (x, y)$, $\dir_0$, $\dir$, and $p$ determines $p'_2$. Again, each of the $K-2$ remaining positions in $\vec{E}$ can be filled in at most $n$ ways, and hence the number of such sequences $\vec{E}$ is $O_{k, \eps}(n^2) \cdot n^{K-2} = O_{k, \eps}(n^K)$.
\end{proof}

\subsection{Initial configuration} In this subsection we construct a no-$(k+1)$-in-line set $S_{\init} \subseteq [n]^2$ of size close to $kn$ that satisfies the pseudorandomness properties required for the completion procedure. It is convenient to introduce the following analogue of \cref{def:absorber}, in which condition \cref{A3} is restricted to $\eps$-heavy lines.

\begin{definition} \label{def:eps-absorber}
  For a set $S \subseteq [n]^2$ and a point $(c, r) \in [n]^2$, we say that a point $(x, y) \in S$ is an \emph{$\eps$-absorber} in $S$ for $(c, r)$ if $(x, r) \notin S$, $(c, y) \notin S$, $(x, y) \in I(c, r)$, and every non-trivial $\eps$-heavy line through $(x, r)$ or $(c, y)$ contains at most $k-1$ points of $S$. Let $\mc{A}_{S, \eps}(p)$ denote the set of $\eps$-absorbers in $S$ for $p$.
\end{definition}

\begin{lemma}
\label{initial-configuration}
    There exists $\gamma = \gamma(k, \eps) > 0$ and a $(2k/n, n^{\gamma})$-spread random set $S_{\init} \subseteq [n]^2$ that always satisfies the following properties:
    \begin{enumerate}
        \item[(a)] $|S_{\init}| = kn - O_k(\eps n)$;
        \item[(b)] $S_{\init}$ contains at most $k$ points on each line;
        \item[(c)] for each $p \in [n]^2$, we have $|\mc{A}_{S_{\init}, \eps}(p)| \ge 2 \rho n$.
    \end{enumerate}
\end{lemma}

\begin{proof}
  We can apply \cref{EGJ-spread} to $\cH^{(k)}_\eps$ with
  \[\Delta = n,\qquad r = |\mc{D}_{\eps}| = O(1/\eps^2),\qquad \delta = 0.1,\qquad L = 2k|\mathcal{D}'_\eps|+1 = O(k/\eps^2),\] and test functions $w_{\size}$, $w_{\rep}$, $w_{\light}$, $(\tilde{w}^{(c, r)}_{\com})_{(c, r) \in [n]^2}$, and $(\tilde{w}^{(c, r)}_{D_1, D_2})_{(c, r) \in [n]^2, D_1, D_2 \subseteq \mc{D}'_{\eps}}$ from \cref{def:test-functions-approximate,def:omega-com,def:omega-D1-D2}. Indeed, we have $\Delta(\cH^{(k)}_{\eps}) = n$ and $\Delta_2(\cH^{(k)}_{\eps}) = 1 \le n^{1-\delta}$ and $|E(\cH^{(k)}_{\eps})| = kn^2 \le \exp(n^{\gamma^2})$ for $\gamma = \delta/(100L^2r^2)$, and the total number of test functions is $3 + n^2 + n^2 2^{O(1/\eps^2)} \le \exp(n^{\gamma^2})$. So, recalling the bounds from \cref{test-functions-approximate-properties,omega-com-properties,omega-D1-D2-properties}, \cref{EGJ-spread} gives us a $(2/n, n^{\gamma})$-spread random matching $\mc{M}$ in $\cH^{(k)}_\eps$ such that  
  \[
  |\mc{M}| = w_{\size}(\mc{M}) = (1 \pm n^{-\gamma}) kn \pm 2n^{\delta}, \qquad w_{\rep}(\mc{M}) = O_k(n^{\delta}), \qquad w_{\light}(\mc{M}) = O_k(\eps n),
  \]
  and for every $(c, r) \in [n]^2$ and $D_1, D_2 \subseteq \mc{D}'_{\eps}$,
  \begin{align}
  w^{(c, r)}_{\com}(\mc{M}) &= \tilde{w}^{(c, r)}_{\com}(\mc{M}) \!\!\!\!\!\!\!\!\!\!\!&&= O_{k, \eps}(n^{\delta}),    \label{eq:omega-com}\\
  w^{(c, r)}_{D_1, D_2}(\mc{M}) &= \tilde{w}^{(c, r)}_{D_1, D_2}(\mc{M}) \!\!\!\!\!\!\!\!\!\!\!&&= \frac{1}{n}\sum_{(x, y) \in I(c, r)}\prod_{\dir_1 \in D_1} (|L^*_{\dir_1}(x, r)|/n)^k\prod_{\dir_2\in D_2} (|L^*_{\dir_2}(c, y)|/n)^k + O_{k,\eps}(n^{1-\gamma}).\label{eq:omega-D1-D2}
  \end{align}
  Let $S_{\mc{M}}$ be the projection of $\mc{M} \subseteq [n]^2 \times [k]$ onto $[n]^2$. Note that $|S_{\mc{M}}| \ge |\mc{M}| - w_{\rep}(\mc{M}) = kn - O_{k}(n^{1-\gamma})$, and, since $\mc{M}$ is a matching in $\cH^{(k)}_{\eps}$, $S_{\mc{M}}$ contains at most $k$ points on each $\eps$-heavy line. Furthermore, since $\mc{M}$ is $(2/n, n^{\gamma})$-spread, $S_{\mc{M}}$ is $(2k/n, n^{\gamma})$-spread. Let $S_{\del}$ be obtained by including one point from each quadruple of points in $S_{\mc{M}}$ that lie on the same $\eps$-light line, and note that $|S_{\del}| \le w_{\light}(\mc{M}) = O_k(\eps n)$. Therefore, $S_{\init} \defeq S_{\mc{M}} \setminus S_{\del} \subseteq S_{\mc{M}}$ is $(2k/n, n^{\gamma})$-spread and satisfies the desired conditions (a) and (b). It remains to check that it satisfies condition (c) as well. 

  For a point $(c, r) \in [n]^2$, define $\mc{A}'_\mc{M}(c, r)$ as the set of $(x, y) \in I(c, r)$ such that $((x, y), 1) \in \mc{M}$ and such that $\mc{M}$ does not contain a set ``obstructing'' a non-trivial $\eps$-heavy direction $\dir$ for a point $p \in \{(x, r), (c, y)\}$: i.e., a set of the form $\{(p'_i, i)\}_{i \in [k]}$ for distinct points $p'_1, \ldots, p'_k \in L^*_{\dir}(p)$ that are not $\eps$-common between $(x, r)$ and $(c, y)$. By \cref{def:omega-D1-D2}, $w^{(c,r)}_{D_1, D_2}(\mc{M})$ counts the number of $(x, y) \in I(c, r)$ such that $((x, y), 1) \in \mc{M}$ and such that $\mc{M}$ contains the sets obstructing all the directions in $D_1$ for $(x, r)$ and all the directions in $D_2$ for $(c, y)$. Hence, by inclusion-exclusion, we have  
  \[
  |\mc{A}'_\mc{M}(c, r)| = \sum_{D_1, D_2 \subseteq \mc{D}'_{\eps}}(-1)^{|D_1|+|D_2|}w^{(c, r)}_{D_1,D_2}(\mc{M}).
  \]
  Combining this with \cref{eq:omega-D1-D2}, we obtain that
  \begin{equation} \label{eq:product-for-absorbers}
    |\mc{A}'_\mc{M}(c, r)| = \frac{1}{n}\sum_{(x, y)\in I(c, r)}\prod_{\dir \in \mc{D}'_{\eps}}(1-(|L^*_{\dir}(x, r)|/n)^k) \prod_{\dir \in \mc{D}'_{\eps}}(1-(|L^*_{\dir}(c, y)|/n)^k) + O_{k,\eps}(n^{1-\gamma}).
  \end{equation}
  To prove a lower bound on this expression, we define 
  \[
  F \defeq \{(x, y) \in [n]^2 : |x-y| \ge 0.001n, |x+y-(n+1)| \ge 0.001n\}. 
  \] 
  Note that for each $p \in F$ and non-trivial direction $\dir$, we have $|L^*_{\dir}(p)| \le 0.999 n$: indeed, this is true for $\dir \in \{(1, 1), (1, -1)\}$ by the definition of $F$, and any other non-trivial direction $\dir$ satisfies $|L^*_{\dir}(p)| \le n/\|\dir\|_{\infty} \le n/2$. Since $1 - x \ge \exp(-10 x)$ for $x \in [0, 0.999]$, this implies that $1 - (|L^*_{\dir}(p)|/n)^k \ge \exp(-10 (|L^*_{\dir}(p)|/n)^k)$.
  Define 
  \[
  I_F(c, r) \defeq \{(x, y) \in I(c, r) : (x, r) \in F, (c, y) \in F\}.
  \]
  Since $[n]^2 \setminus F$ contains at most $0.004 n$ points in each column or row, \cref{irrelevant-lines} implies that $|I_F(c, r)| \ge |I(c, r)| - 0.008 n^2 \ge 0.01 n^2$. Therefore, we can lower bound the right-hand side of \cref{eq:product-for-absorbers} as follows:
  \begin{align*}
  |\mc{A}'_{\mc{M}}(c, r)| &\ge \frac{1}{n} \sum_{(x, y) \in I_F(c, r)} \exp\left(-10 \sum_{\dir \in \mc{D}'_{\eps}} \left((|L^*_{\dir}(x, r)|/n)^k + (|L^*_{\dir}(c, y)|/n)^k\right)\right) + O_{k, \eps}(n^{1-\gamma}) \\
  &\ge 0.01 n \cdot \exp\left(-20 \sum_{\dir \in \mc{D}'_{\eps}} \frac{1}{\|\dir\|_{\infty}^k}\right) + O_{k, \eps}(n^{1-\gamma}) \ge 0.01 n \cdot \exp\left(-20 \sum_{m = 1}^n \frac{4m}{m^3}\right) + O_{k, \eps}(n^{1-\gamma}) = \Omega(n),
  \end{align*}
  where in the last inequality we used that $k \ge 3$ and that for each $m \in \NN$ there are at most $4m$ directions $\dir$ with $\|\dir\|_{\infty} = m$.

  Suppose that $(x, y) \in \mc{A}'_\mc{M}(c, r)$ but $(x, y)$ is not an $\eps$-absorber in $S_{\init}$ for $(c, r)$. Then one of the following must hold:
  \begin{itemize}
    \item $(x, y) \in S_{\del}$;
    \item $(x, r) \in S_{\mc{M}}$ or $(c, y) \in S_{\mc{M}}$;
    \item some point $p' \in S_{\mc{M}}$ is $\eps$-common between $(x, r)$ and $(c, y)$.
  \end{itemize}
  The number of points satisfying the first item is at most $|S_{\del}| = O_k(\eps n)$. Since $S_{\mc{M}}$ contains at most $k$ points in each column or row, the number of points satisfying the second item is at most $2k$. By \cref{eq:omega-com}, the number of points satisfying the third item is $O_{k, \eps}(n^{\delta})$. Therefore, by our choice of $\rho$ and $\eps$ from \cref{eq:dependencies-of-parameters},  
  \[
  |\mc{A}_{S_{\init}, \eps}(c, r)| \ge |\mc{A}'_{\mc{M}}(c, r)| - O_k(\eps n) - 2k - O_{k, \eps}(n^{\delta}) \ge 2 \rho n. \qedhere
  \]
\end{proof}

\subsection{Randomised completion procedure} 

In this subsection, we describe a randomised procedure that completes the (random) initial configuration $S_{\init}$ given by \cref{initial-configuration} into a no-$(k+1)$-in-line set of size $kn$ with high probability. 

\begin{definition}[Completion procedure] \label{def:completion-procedure}
Let $T \defeq kn - |S_{\init}|$, and note that $T = O_k(\eps n)$ by \cref{initial-configuration}(a). Let $(c_1, \ldots, c_T)$ and $(r_1, \ldots, r_T)$ be the sequences of indices of ``unsaturated'' columns and rows (in arbitrary order), where the number of times each column/row index appears equals $(k - \text{\#\{points of $S_{\init}$ in that column/row\}})$. Consider the following randomised algorithm:  
\begin{algorithm}[H]
\caption{}
\label{algo:completion}
\begin{algorithmic}[1]
\State $S_0 \gets S_{\init}$
\For{$t = 1$ to $T$}
  \If{$\mathcal{A}_{S_{t-1}}(c_t, r_t) = \emptyset$}
    \State Abort
  \EndIf    
  \State Pick $(x_t, y_t) \in \mathcal{A}_{S_{t-1}}(c_t, r_t)$ uniformly at random
  \State $S_t \gets \bigl(S_{t-1} \setminus \{(x_t, y_t)\}\bigr) \cup \{(x_t, r_t), (c_t, y_t)\}$
\EndFor
\State \Return $S_T$
\end{algorithmic}
\end{algorithm}
\end{definition}
Note that each successful step of this algorithm increases the number of points of the current set in the target column $c_t$ and row $r_t$ by one, while keeping the number of its points in every other row and column unchanged. Thus, by the choice of $(c_1, \ldots, c_T)$ and $(r_1, \ldots, r_T)$ and the definition of absorbers (\cref{def:absorber}), each step preserves the no-$(k+1)$-in-line property of the current set and increases its size by one. In the next subsection, we will prove the following proposition stating that, with high probability (taking into account both the randomness of $S_{\init}$ and the random choices made by the algorithm), this algorithm does not abort.

\begin{proposition} \label{many-absorbers} 
  With probability at least $1 - \exp(-n^{\Omega_{k, \eps}(1)})$, \cref{algo:completion} does not abort, and, furthermore, for every $t \in [T]$ and $(c, r) \in [n]^2$ we have
  \[
  |\mc{A}_{S_{t-1}}(c, r)| \ge \rho n.
  \]
\end{proposition}

If this event occurs, then the final set $S_T$ has size $kn$ and contains at most $k$ points on each line. Recalling the discussion at the beginning of this section, we conclude that \cref{many-absorbers} implies \cref{no-k+1-in-line-kn}.

\subsection{Analysing the completion procedure via spreadness} To prove \cref{many-absorbers}, we will bound the number of points $(x, y)$ which are $\eps$-absorbers in $S_{\init}$ (for some point $(c, r)$), but are not  $\eps$-absorbers at some later step of the completion procedure. The main reason why this can happen is that $(x,y)$ may violate \cref{A3} in \cref{def:absorber}: there may be a line through $(x, r)$ or $(c, y)$ containing at least $k$ points. It is easy to bound the impact of the $\alpha$-heavy lines: since $\eps \le \alpha$, an $\alpha$-heavy line through $(x, r)$ or $(c, y)$ can only contain $k$ points if some point was added to this line in the completion phase itself, and this can happen for at most $O_{k}((\eps/\alpha^2) n)$ different $\eps$-absorbers. Handling the $\alpha$-light lines is more difficult, and this is where we will use the spreadness of $S_{\init}$ (via \cref{light-lines-via-spread}).

\begin{definition}
  For a column or row $L$ of the grid $[n]^2$, let $\mc{T}_L$ be the set of triples of points in $[n]^2 \setminus L$ that lie on the same $\alpha$-light line intersecting $L$; that is, 
  \[
  \mc{T}_L \defeq \Big\{\{p_1, p_2, p_3\} \in \binom{[n]^2 \setminus L}{3} : \text{$p_1, p_2, p_3 \in L^*_{\dir}(p)$ for some $\alpha$-light direction $\dir$ and a point $p \in L$}\Big\}.
  \]
  For a set $S \subseteq [n]^2$, let $\mc{T}_L(S) \defeq \{T \in \mc{T}_L : T \subseteq S\}$.
\end{definition}

\begin{lemma} \label{light-lines-via-spread}
  Let $L$ be a column or row of the grid, and let $S \subseteq [n]^2$ be a $(c/n, 3m)$-spread random set for some $c \ge 1$ and a positive integer $m \le (\alpha n)^{1/3}$. Then, with probability at least $1 - \exp(-m)$, 
  \[  
  \left|\mc{T}_L(S)\right| = O(c^3 \alpha n).
  \]
\end{lemma}

\cref{light-lines-triples} implies that for each point $p \in [n]^2$ there are $O(\alpha n^3)$ triples of points $\{p_1, p_2, p_3\} \subseteq [n]^2 \setminus \{p\}$ that lie on the same $\alpha$-light line passing through $p$. Here we also need a similar bound for pairs instead of triples.

\begin{lemma} \label{light-lines-pairs}
  For every point $p \in [n]^2$, there are $O(n^2 \log n)$ pairs of points $\{p_1, p_2\} \subseteq [n]^2 \setminus \{p\}$ that lie on the same line passing through $p$.
\end{lemma}
\begin{proof}
  For each $m \in \NN$, there are at most $4m$ directions $\dir \in \mc{D}$ with $\|\dir\|_\infty = m$. For each direction $\dir \in \mc{D}$, we have $|L^*_{\dir}(p)| \le n/\|\dir\|_\infty$, and hence there are at most $(n/\|\dir\|_\infty)^2$ choices for the pair $\{p_1, p_2\}$ on this line. Taking the sum over all directions $\dir \in \mc{D}$, we conclude that the total number of such pairs is at most
  \[
  n^2 \sum_{\dir \in \mc{D}} \frac{1}{\|\dir\|_\infty^2} \le n^2 \sum_{m = 1}^n \frac{4m}{m^2} = O(n^2 \log n). \qedhere
  \]
\end{proof}

\begin{proof}[\textbf{Proof of \cref{light-lines-via-spread}}]
  Write $q \defeq c/n$. For a set of points $P \subseteq [n]^2$, let 
  \[
    W(P) \defeq \sum_{T \in \mc{T}_L} q^{|T \setminus P|}.
  \]
  Then, by $(q, 3m)$-spreadness of $S$,  
  \[
  \EE[|\mc{T}_L(S)|^m] \le \sum_{T_1, \ldots, T_m \in \mc{T}_L} \PP[T_1 \cup \ldots \cup T_m \subseteq S] \le \sum_{T_1, \ldots, T_m \in \mc{T}_L} q^{|\bigcup_{i=1}^m T_i|}.
  \]
  This expression can be bounded inductively as
  \begin{equation*} \begin{aligned}
  \sum_{T_1, \ldots, T_m \in \mc{T}_L} q^{|\bigcup_{i=1}^m T_i|} &\le \sum_{T_1, \ldots, T_{m-1} \in \mc{T}_L} \left(q^{|\bigcup_{i=1}^{m-1} T_i|} \sum_{T_m \in \mc{T}_L} q^{|T_m \setminus \bigcup_{i=1}^{m-1} T_i|}\right) \\
  &\le \left(\max_{|P| \le 3m} W(P)\right) \cdot \sum_{T_1, \ldots, T_{m-1} \in \mc{T}_L} q^{|\bigcup_{i=1}^{m-1} T_i|} \le \left(\max_{|P| \le 3m} W(P)\right)^m.
  \end{aligned} \end{equation*}
  Therefore, $\EE[|\mc{T}_L(S)|^m] \le \left(\max\limits_{|P| \le 3m} W(P)\right)^m$. To estimate $W(P)$, we note that 
  \begin{equation} \label{eq:W(P)-bound}
  W(P) \le q^3 W_0 + q^2 |P| W_1 + q |P|^2 W_2 + |P|^3 W_3, \quad \text{ where } \quad
  W_j = \max_{|X| = j} \big|\big\{T \in \mc{T}_L : X \subseteq T\big\}\big|.
  \end{equation}
  Summing the bound given by \cref{light-lines-triples} over all points $p \in L$, we obtain that $W_0 = O(\alpha n^4)$. \cref{light-lines-pairs} implies that $W_1 = O(n^2 \log n)$. Since a pair of points determines a line, we have $W_2 \le n$. Finally, we trivially have $W_3 \le 1$. Recalling that $c \ge 1$ and $m \le (\alpha n)^{1/3}$, we substitute these bounds into \cref{eq:W(P)-bound} and obtain that for every set $P \subseteq [n]^2$ of size at most $3m$,
  \[
  W(P) = O\left(c^3 \alpha n + c^2 m \log n + c m^2 + m^3\right) = O(c^3 \alpha n).
  \]
  Therefore, $\EE[|\mc{T}_L(S)|^m] \le \left(C \cdot c^3 \alpha n\right)^m$ 
  for some absolute constant $C$, and by Markov's inequality,
  \[
  \PP[|\mc{T}_L(S)| > 3C \cdot c^3 \alpha n] \le \frac{\EE[|\mc{T}_L(S)|^m]}{(3C \cdot c^3 \alpha n)^m} \le \exp(-m). \qedhere
  \] 
\end{proof}

\begin{definition}[Stopping time]
  Recall the setup in \cref{def:completion-procedure}. Define $\tau$ as the minimal $t \in [T]$ such that $|\mc{A}_{S_{t-1}}(c, r)| < \rho n$ for some $(c, r) \in [n]^2$ (i.e., the event in \cref{many-absorbers} does not occur), or $\tau = T + 1$ if no such $t$ exists. Let $S_{\comp}$ be the set of points added by \cref{algo:completion} up to step $\tau$:
  \[
  S_{\comp} \defeq \bigcup_{t < \tau} \{(x_t, r_t), (c_t, y_t)\}.
  \]
\end{definition}
In these terms, we need to show that $\tau = T + 1$ with high probability. One could attempt to approach this by applying \cref{light-lines-via-spread} to the set $S_{\init} \cup S_{\comp} \supseteq S_{\tau-1}$. Unfortunately, since $S_{\comp}$ is constructed by adding two points at a time, it does not satisfy the required spreadness condition. To avoid this issue, we instead work with the set $S_{\half}$ obtained by including a uniformly random one of the two points at each step.

\begin{definition}
  For each $1 \le t < \tau$, let $p_{\half}(t)$ be either $(x_t, r_t)$ or $(c_t, y_t)$ with probability $1/2$, independently of each other and of all other random choices in our procedure. Then, let $S_{\half} \defeq \{p_{\half}(t) : 1 \le t < \tau\}$.
\end{definition}

\begin{lemma} \label{spread-half}
If we condition on any outcome of $S_{\init}$, then the conditional distribution of $S_{\half}$ is $(4k^2/(\rho n), n^2)$-spread.
\end{lemma}

\begin{proof}
  All probabilities in this proof are assumed to be conditional on an arbitrary fixed outcome of $S_{\init}$. Consider a set of points $P = \{p_1, \ldots, p_N\} \subseteq [n]^2$. If $P \subseteq S_{\half}$, then there are \emph{distinct} steps $t_1, \ldots, t_N < \tau$ such that $p_i = p_{\half}(t_i)$ for each $i \in [N]$. Denote this event by $\mc{E}(t_1, \ldots, t_N)$. If it occurs, then for each $i \in [N]$ we have $p_i = (x_{t_i}, r_{t_i})$ or $p_i = (c_{t_i}, y_{t_i})$, and hence $p_i$ shares a column or a row both with the point $(c_{t_i}, r_{t_i})$ and the absorber $(x_{t_i}, y_{t_i})$ in $S_{t_i-1}$ for $(c_{t_i}, r_{t_i})$.
  
  By the definition of $\tau$, the condition $t_i < \tau$ implies that at step $t_i$ we have at least $\rho n$ absorbers to choose from. Since $S_{t_i-1}$ contains at most $k$ points in each column or row, at most $2k$ of them share a column or row with $p_i$. Therefore, 
  \begin{equation} \label{eq:probability-of-event-E}
  \PP[\mc{E}(t_1, \ldots, t_N)] \le \EE\left[\prod_{i = 1}^{N} \frac{\mathbf{1}_{t_i < \tau} \cdot 2k}{|\mc{A}_{S_{t_i-1}}(c_{t_i}, r_{t_i})|}\right] \le (2k/(\rho n))^{N}.
  \end{equation}
  On the other hand, the sequences $(c_1, \ldots, c_T)$ and $(r_1, \ldots, r_T)$ each contain at most $k$ copies of each column/row index, and hence for each $i \in [N]$ there are at most $2k$ values of $t_i$ such that $\mc{E}(t_1, \ldots, t_N)$ can possibly occur. So, taking the sum of \cref{eq:probability-of-event-E} over at most $(2k)^N$ possible sequences of times $t_1, \ldots, t_N$, we conclude that 
  $\PP[P \subseteq S_{\half}] \le (4k^2/(\rho n))^N$, as desired.
\end{proof}

\begin{lemma} \label{few-bad-triples}
  Let $L$ be a column or row of the grid. With probability at least $1 - \exp(-n^{\Omega_{k, \eps}(1)})$, we have 
  \[
  |\mc{T}_L(S_{\init} \cup S_{\comp})| = O_k\left((\alpha/\rho^3) n\right).
  \]
\end{lemma}

\begin{proof}
  By \cref{initial-configuration}, $S_{\init}$ is $(2k/n, n^{\gamma})$-spread for some $\gamma = \gamma(k, \eps) > 0$. By \cref{spread-half}, $S_{\half}$ is $(4k^2/(\rho n), n^2)$-spread, even after conditioning on an arbitrary outcome of $S_{\init}$. Hence, by \cref{spread-union}, $S_{\init} \cup S_{\half}$ is $(8k^2/(\rho n), n^{\gamma})$-spread. Applying \cref{light-lines-via-spread} with $m = n^{\gamma}/3$ to a column or row $L$, we conclude that with probability at least $1 - \exp(-n^{\gamma}/3)$
  \begin{equation} \label{eq:bound-for-half}
  |\mc{T}_L(S_{\init} \cup S_{\half})| \le (C \alpha/\rho^3) n
  \end{equation}
  for some $C = C(k)$. 

  It remains to ``transfer'' this bound to $S_{\init} \cup S_{\comp}$. Fix a certain outcome of $S_{\init}$ and $S_{\comp}$, and consider a triple of points $\{p_1, p_2, p_3\} \in \mc{T}_L(S_{\init} \cup S_{\comp})$. Then, the line containing $p_1, p_2, p_3$ also contains some point of $L$, and thus is not irrelevant. Hence, by \cref{A2} (from \cref{def:absorber}), it contains at most one point from each pair of the form $\{(x_t, r_t), (c_t, y_t)\}$ for $t < \tau$. So, conditionally on this outcome of $S_{\init}$ and $S_{\comp}$ as well as the choices of absorbers made during the completion procedure, each of the points $p_1, p_2, p_3$ is included in $S_{\init} \cup S_{\half}$ with probability at least $1/2$ independently of each other. Taking the sum over all triples (and averaging over the choices of absorbers), we obtain that $\EE\big[|\mc{T}_L(S_{\init} \cup S_{\half})| \,\big|\, S_{\init}, S_{\comp}\big] \ge |\mc{T}_L(S_{\init} \cup S_{\comp})|/8$, and consequently,
  \[
  \PP\Big[|\mc{T}_L(S_{\init} \cup S_{\half})| \ge \frac{1}{16}|\mc{T}_L(S_{\init} \cup S_{\comp})| \;\Big|\; S_{\init}, S_{\comp}\Big] \ge \frac{1}{16}.
  \]
  Denoting the event that $|\mc{T}_L(S_{\init} \cup S_{\comp})| > (16 C \alpha / \rho^3) n$ by $\mc{E}$, we have
  \begin{align*}
  \PP\Big[|\mc{T}_L(S_{\init} \cup S_{\half})| > (C \alpha / \rho^3) n\Big] \ge \EE\left[\mathbf{1}_{\mc{E}} \cdot \PP\Big[|\mc{T}_L(S_{\init} \cup S_{\half})| \ge \frac{1}{16}|\mc{T}_L(S_{\init} \cup S_{\comp})| \;\Big|\; S_{\init}, S_{\comp}\Big] \right].
  \end{align*}
  Together with \cref{eq:bound-for-half} this yields that $\PP[\mc{E}] \le 16 \exp(-n^{\gamma}/3)$, and thus $|\mc{T}_L(S_{\init} \cup S_{\comp})| \le (16 C \alpha / \rho^3) n$ with probability at least $1 - 16 \exp(-n^{\gamma}/3)$, as desired.
\end{proof}

\begin{proof}[\textbf{Proof of \cref{many-absorbers}}]
  By \cref{few-bad-triples} (and the union bound over $2n$ columns and rows), with probability at least $1 - \exp(-n^{\Omega_{k, \eps}(1)})$,  every column or row $L$ of the grid satisfies $|\mc{T}_L(S_{\init} \cup S_{\comp})| = O_k((\alpha/\rho^3) n)$. So, it suffices to show that if this event occurs, then $\tau = T + 1$.
  
  Suppose that $\tau = t \le T$, and thus $S_{t-1} \subseteq S_{\init} \cup S_{\comp}$. Fix an arbitrary point $(c, r) \in [n]^2$, and recall that $|\mc{A}_{S_{\init}, \eps}(c, r)| \ge 2 \rho n$ by \cref{initial-configuration}(c). If a ``candidate'' $\eps$-absorber $(x, y)$ in $S_{\init}$ for $(c, r)$ is not an actual absorber in $S_{t-1}$ for $(c, r)$, then one of the following must hold:
  \begin{itemize}
  \item \textbf{Case 1: $(x, y)$ does not lie in $S_{t-1} \supseteq S_{\init} \setminus \{(x_1, y_1), \ldots, (x_{t-1}, y_{t-1})\}$.} Since $t \le T = O_k(\eps n)$, this can affect at most $O_k(\eps n)$ candidates $(x, y)$;
  \item \textbf{Case 2: either $(x, r)$ or $(c, y)$ is already in $S_{t-1}$.}
  Since both $S_{t-1}$ and $S_{\init}$ contain at most $k$ points in each column or row, this affects at most $2k^2$ candidates $(x,y)$. 
  \item \textbf{Case 3: there is a non-trivial line $L$ through either $(x, r)$ or $(c, y)$ that contains at least $k$ other points of $S_{t-1}$.} 
  \begin{itemize}[leftmargin=8pt]
    \item \textbf{Case 3a: $L$ is $\alpha$-heavy.} Since $\eps \le \alpha$ and $(x, y)$ is an $\eps$-absorber in $S_{\init}$ for $(c, r)$, at least one of these $k$ points was added during the first $t-1$ steps of the completion procedure. Since there are $|\mc{D}'_{\alpha}| = O(1/\alpha^2)$ different $\alpha$-heavy lines through each of $O_k(\eps n)$ such points, the number of possible lines $L$ is $O_k((\eps/\alpha^2) n)$. Therefore, since $S_{\init}$ contains at most $k$ points in each column or row, this affects $O_k((\eps/\alpha^2)  n)$ candidates $(x, y)$. 

    \item \textbf{Case 3b: $L$ is $\alpha$-light.} Let $L_1$ and $L_2$ be the column and row of $(c, r)$, respectively. Since $k \ge 3$ and $S_{t-1} \subseteq S_{\init} \cup S_{\comp}$, the line $L$ must contain a triple of points in $\mc{T}_{L_1}(S_{\init} \cup S_{\comp})$ or $\mc{T}_{L_2}(S_{\init} \cup S_{\comp})$. Hence, the total number of points $(c, y) \in L_1$ and $(x, r) \in L_2$ blocked by such lines is at most
    \[
    |\mc{T}_{L_1}(S_{\init} \cup S_{\comp})| + |\mc{T}_{L_2}(S_{\init} \cup S_{\comp})| = O_k((\alpha/\rho^3) n).
    \]
    Again, since $S_{\init}$ contains at most $k$ points in each column or row, this affects $O_k((\alpha/\rho^3) n)$ candidates $(x, y)$. 
  \end{itemize}
\end{itemize}
Taking the sum over all these cases, we conclude that for each point $(c, r) \in [n]^2$
\[
|\mc{A}_{S_{t-1}}(c, r)| \ge 2 \rho n - O_k(\eps n) - 2k^2 - O_k((\eps/\alpha^2) n) - O_k((\alpha/\rho^3) n), 
\]
which is at least $\rho n$ by our choice of parameters from \cref{eq:dependencies-of-parameters}. However, by the definition of $\tau$, this implies that in fact $\tau > t$, contradicting our assumption that $\tau = t$.
\end{proof}

\section{Higher dimensions}
\label{sec:higher-dimensions}

The proof of \cref{higher-dimensions} is largely similar to the proof of \cref{no-k+1-in-line-approximate} in \cref{sec:approximate}. The main difference is that if one applies \cref{EGJ-spread} to the test function $w_{\light}$ defined in an analogous way, then the resulting bound \cref{eq:EGJ} would not be strong enough (because $B_{\Delta}(w)$ would be large compared to $w(E(\cH))$, at least for some values of $d$ and $s$). Instead, we do not work with $w_{\light}$ as a test function directly, but deduce a sufficient bound from the spreadness of our matching. 

Let $\mc{D}^{d, s}$ be the set of primitive\footnote{A lattice $\Gamma \subseteq \ZZ^d$ is called primitive if $\operatorname{span}_{\RR}(\Gamma) \cap \ZZ^d = \Gamma$.} rank-$s$ lattices $\Gamma \subseteq \ZZ^d$ such that $\Gamma \cap [-n, n]^d$ contains $s$ linearly independent vectors. For each $\eps > 0$, we say that a lattice $\Gamma \in \mc{D}^{d, s}$ is $\eps$-heavy if $1/\det\Gamma \ge \eps$, and $\eps$-light otherwise. Let $\mc{D}^{d, s}_{\eps} \defeq \{\Gamma \in \mc{D}^{d, s} : 1/\det\Gamma \ge \eps\}$ be the set of $\eps$-heavy lattices. We remark that lattices here play the same role as directions did in \cref{sec:approximate}, and $\det\Gamma$ is a suitable analogue of $\|\dir\|_{\infty}$ for this setting\footnote{Determinant $\det \Gamma$ of a rank-$s$ lattice $\Gamma$ equals $\sqrt{\det G}$, where $G = (\langle v_i, v_j \rangle)_{1 \le i, j \le s}$ is the Gram matrix of some generating vectors $v_1, \ldots, v_s$ of $\Gamma$.}. 

By a result of Schmidt \cite[Theorem~2]{schmidt-68}, the number of lattices in $\ZZ^d$ with determinant at most $M$ is $O_d(M^d)$, which immediately implies that 
\begin{equation} \label{eq:counting-lattices}
|\mc{D}^{d, s}_{\eps}| = O_d(1/\eps^d).
\end{equation}
Note that for each affine subspace $V$ of dimension $s$, the intersection $V \cap [n]^d$ is contained in a translate of some lattice $\Gamma \in \mc{D}^{d, s}$. Indeed, by adding more points if necessary, we can find a set $X$ such that $V \cap [n]^d \subseteq X \subseteq [n]^d$ and the affine span $V'$ of $X$ has dimension $s$. Then, $V' \cap \ZZ^d$ is a translate of a primitive rank-$s$ lattice with $s$ linearly independent vectors in $[-n, n]^d$.

\begin{lemma} \label{counting-lattice-points}
  For every lattice $\Gamma \in \mc{D}^{d, s}$ and every point $p \in [n]^d$, we have $|(p + \Gamma) \cap [n]^d| = O_d(n^s / \det\Gamma)$.
\end{lemma}
\begin{proof}
  Let $V$ be the linear span of $\Gamma$. Note that $\dim V = s$, and that the $s$-dimensional volume of $[-n, n]^d \cap V$ is $O_d(n^s)$. Since $\Gamma \cap [-n, n]^d$ contains $s$ linearly independent vectors, it does not lie in any proper subspace of $V$. Thus, a result of Widmer \cite[Corollary 2.10]{widmer-12} (applied to $\Gamma$ as a lattice in $V$) implies that 
  \[
  |(p + \Gamma) \cap [n]^d| \le |\Gamma \cap [-n, n]^d| = O_d(n^s / \det(\Gamma)). \qedhere
  \]
\end{proof}

\begin{lemma} \label{counting-tuples-in-light-lattices}
  The number of sets in $[n]^d$ of size $d+2$ contained in some translate of an $\eps$-light lattice in $\mc{D}^{d, s}$ is $O_d(\eps n^{d + s(d+1)})$. 
\end{lemma}
\begin{proof}
  Fix a lattice $\Gamma \in \mc{D}^{d, s}$ and choose one point $p_1$ of the set. Then, each of the remaining points $p_2, \ldots, p_{d+2}$ must lie in $(p_1 + \Gamma) \cap [n]^d$, which has size at most $C n^s / \det\Gamma$ for some $C = C(d)$ by \cref{counting-lattice-points}. Combining this with \cref{eq:counting-lattices}, we conclude that the total number of such sets is at most
  \begin{align*}
  n^d \sum_{\substack{\Gamma \in \mc{D}^{d, s} \\ 1/\det \Gamma < \eps}} (C n^s/\det\Gamma)^{d+1} &= C^{d+1} n^{d+s(d+1)}\sum_{j = 0}^{\infty}\sum_{\substack{\Gamma \in \mc{D}^{d, s} \\ 1/\det \Gamma \in [\eps 2^{-j-1}, \eps 2^{-j})}} (1/\det\Gamma)^{d+1} \\
  &\le C^{d+1} n^{d+s(d+1)} \sum_{j = 0}^{\infty} |\mc{D}^{d, s}_{\eps2^{-j-1}}| \cdot (\eps 2^{-j})^{d+1} = O_d(\eps n^{d+s(d+1)}). \qedhere
  \end{align*}
\end{proof}

\begin{proof}[\textbf{Proof of \cref{higher-dimensions}}]
  By \cite[Theorem~1.3]{grebennikov-kwan-25}, there exists $K = K(d, \eta)$ such that the desired result holds true whenever $k \ge K$. Thus, we focus on the case $d+1 \le k < K$. In particular, we may assume that $n$ is sufficiently large in terms of $k$ (in addition to $d$ and $\eta$).

  Let $\eps \defeq c \eta$ for a sufficiently small constant $c = c(d, k) > 0$. Consider the following $|\mc{D}^{d, s}_{\eps}|$-uniform $|\mc{D}^{d, s}_{\eps}|$-partite hypergraph $\cH_{\eps}$: the vertices of $\mc H_\eps$ are the translates of $\eps$-heavy lattices which intersect $[n]^d$, and for each $p \in [n]^d$ we put an edge $\{p + \Gamma : \Gamma \in \mc{D}^{d, s}_{\eps}\}$.
  Let $\cH^{(k)}_{\eps}$ be the union of $k$ disjoint copies of $\cH_{\eps}$. We identify the edge set of $\cH^{(k)}_{\eps}$ with $[n]^d \times [k]$. Note that an affine subspace of dimension $s_0$ contains at most $n^{s_0}$ points of $[n]^d$. So, $\Delta(\cH^{(k)}_{\eps}) \le n^s$, and, since an intersection of two primitive rank-$s$ lattices has rank at most $s-1$, we have $\Delta_2(\cH^{(k)}_{\eps}) \le n^{s-1}$.

  As before, let $w_{\size}$ be the $1$-uniform test function such that $w_{\size}(e) = 1$ for each $e \in E(\cH^{(k)}_{\eps})$, and let $w_{\mathrm{rep}}:\binom{E(\cH^{(k)}_{\eps})}{2} \to \{0,1\}$ be the $2$-uniform test function such that $w_{\mathrm{rep}}(\{(p_1, i_1), (p_2, i_2)\}) = 1$ if and only if $p_1 = p_2$. Clearly, 
  \[
  w_{\size}(E(\cH^{(k)}_{\eps})) = k n^d, \quad B_{n^s}(w_{\size}) = n^{s}, \quad w_{\mathrm{rep}}(E(\cH^{(k)}_{\eps})) = \binom{k}{2} n^d, \quad B_{n^s}(w_{\mathrm{rep}}) = n^{2s}.
  \]
  Applying \cref{EGJ-spread} to $\cH^{(k)}_{\eps}$ with $\Delta = n^s$ and $r = |\mc{D}^{d, s}_{\eps}| = O_d(1/\eps^d)$ and $\delta = 1/d$ and $L = 2$, and test functions $w_{\size}$ and $w_{\mathrm{rep}}$, we obtain a $(2/n^s, d+2)$-spread matching $\mc{M}$ in $\cH^{(k)}_{\eps}$ such that 
  \[
  |\mc{M}| = w_{\size}(\mc{M}) = (1 + o(1)) k n^{d-s} \quad \text{ and } \quad w_{\mathrm{rep}}(\mc{M}) = O_k(n^{\max(d-2s, s/d)}) = o(n^{d-s}).
  \] 
  Viewing $\mc{M}$ as a subset of $[n]^d \times [k]$ and taking the projection onto the first coordinate, we obtain a $(2k/n^s, d+2)$-spread set $S_{\mc{M}} \subseteq [n]^d$ of size at least $|\mc{M}| - w_{\mathrm{rep}}(\mc{M}) = (1 + o(1)) k n^{d-s}$ that contains at most $k$ points in each translate of an $\eps$-heavy lattice. 

  Let $\mc{C}$ be the collection of subsets of $[n]^d$ of size $d+2$ contained in a translate of some $\eps$-light lattice. By \cref{counting-tuples-in-light-lattices}, $|\mc{C}| = O_d(\eps n^{d + s(d+1)})$. Let $S_{\del}$ be the subset of $S_{\mc{M}}$ obtained by including one point from each set $C \in \mc{C}$ contained in $S_{\mc{M}}$. Then, 
  \[
  \EE[|S_{\del}|] \le \sum_{C \in \mc{C}} \PP[C \subseteq S_{\mc{M}}] \le |\mc{C}| \cdot (2k/n^s)^{d+2} = O_{d, k}(\eps n^{d-s}),
  \]
  and hence there exists an outcome of $S_{\mc{M}}$ such that $|S_{\mc{M}} \setminus S_{\del}| \ge (1+o(1)) k n^{d-s} - O_{d, k}(\eps n^{d-s})$, which is at least $(1-\eta) k n^{d-s}$ by our choice of $\eps$. Since $k \ge d+1$, $S_{\mc{M}} \setminus S_{\del}$ contains at most $k$ points in each $s$-dimensional affine subspace by construction, completing the proof.
\end{proof}

\section{No-four-on-a-circle problem}
\label{sec:no-four-on-a-circle}

\subsection{Counting tools}

As noted in \cite{ghosal-goenka-keevash-25}, it follows from the work of Huxley and Konyagin \cite{huxley-konyagin-09} that most cyclic quadrilaterals in $[n]^2$ are isosceles trapezia\footnote{For us, an isosceles trapezium is a quadrilateral with a pair of parallel sides that share a common perpendicular bisector.}.

\begin{lemma}[{see \cite[Lemma 4.2]{ghosal-goenka-keevash-25}}] \label{asymmetric-count}
  The number of cyclic quadrilaterals in $[n]^2$ that are not isosceles trapezia is at most $n^{4 + \frac{18}{29} + o(1)}$ as $n \to \infty$.
\end{lemma}

Recall that $\mc{D}$ is the set of possible line directions, and $\mc{D}_{\eps} = \{\dir \in \mc{D} : 1/\|\dir\|_{\infty} \ge \eps\}$ is the set of $\eps$-heavy directions. Also, for a point $p \in \RR^2$, we write $L_{\dir}(p)$ for the intersection of the grid $[n]^2$ with the line through $p$ in direction $\dir$, and let $L^*_{\dir}(p) = L_{\dir}(p) \setminus \{p\}$. For each $\dir \in \mc{D}$ and $\eps \in (0, 1)$, define
\[
\mc{L}_{\dir} \defeq \{L_{\dir}(p) : p \in [n]^2\}, \qquad \mc{L}_{\eps} \defeq \bigsqcup_{\dir \in \mc{D}_{\eps}} \mc{L}_{\dir}.
\] 
We say that a line is $\eps$-heavy (resp. $\eps$-light) if its direction is $\eps$-heavy (resp. $\eps$-light). Similarly, we say that an isosceles trapezium is $\eps$-heavy (resp. $\eps$-light) if its parallel sides\footnote{Note that if a trapezium is a rectangle then both its pairs of parallel sides are $\eps$-heavy or $\eps$-light simultaneously.} are $\eps$-heavy (resp. $\eps$-light).

\begin{lemma} \label{light-trapezia}
  For every $\eps > 0$, the number of $\eps$-light isosceles trapezia in $[n]^2$ is $O(\eps n^5)$.
\end{lemma}
\begin{proof}
  First, we bound the number of isosceles trapezia with parallel sides in a given direction $\dir$. Each such trapezium is determined by the choice of the midpoints $p_1$ and $p_2$ of its parallel sides (note that they must be contained in the half-integer grid $\{(a/2, b/2) : a, b \in [2n]\}$ and lie on the same line in direction $\dir^{\perp} = (-b, a)$), and the choice of one of the vertices on each of the two parallel sides $v_1 \in L^*_{\dir}(p_1)$, $v_2 \in L^*_{\dir}(p_2)$. Then, we have at most $(2n)^2$ choices for $p_1$, at most $2n/\|\dir\|_{\infty}$ choices for $p_2$, and at most $n/\|\dir\|_{\infty}$ choices for each of $v_1$ and $v_2$. Therefore, the number of such trapezia is $O(n^5/\|\dir\|_{\infty}^3)$. 

  For each $m \in \NN$, the number of directions $\dir \in \mc{D}$ with $\|\dir\|_{\infty} = m$ is at most $4m$. Hence, taking the sum over all $\eps$-light directions, we conclude that the number of $\eps$-light isosceles trapezia is bounded by
  \[
  \sum_{\dir \in \mc{D} \setminus \mc{D}_{\eps}} O(n^5/\|\dir\|_{\infty}^3) = O\left(n^5 \sum_{m = \lfloor 1/\eps \rfloor + 1}^{n} \frac{4m}{m^3}\right) = O(\eps n^5). \qedhere
  \]
\end{proof}

For a pair of points $p_1, p_2 \in [n]^2$, let $\ell^{\perp}(p_1, p_2)$ be the line through the midpoint of the segment $p_1 p_2$ orthogonal to this segment (we later refer to such lines as \emph{bisectors}). Note that a set $S \subseteq [n]^2$ with no four points on a line contains no isosceles trapezia if and only if for all pairs of distinct points $p_1, p_2 \in S$, the bisectors $\ell^{\perp}(p_1, p_2)$ are different. For a direction $\dir \in \mc{D}$ and $\eps \in (0, 1)$, define 
\[
\mc{L}^{\perp}_{\dir} \defeq \{\ell^{\perp}(p_1, p_2) : p_1, p_2 \in [n]^2,\; p_2 \in L^*_{\dir}(p_1)\}, \qquad \mc{L}^{\perp}_{\eps} \defeq \bigsqcup_{\dir \in \mc{D}_{\eps}} \mc{L}^{\perp}_{\dir}.
\]
For a bisector $\ell \in \mc{L}^{\perp}_{\eps}$ and a point $p \in \RR^2$, let $R_{\ell}(p)$ be the point obtained by reflecting $p$ across $\ell$.

\subsection{Proof of \cref{no-four-on-a-circle-2n}} As discussed in the introduction (\cref{rmk:no-4-on-a-circle}), our proof combines two consecutive applications of \cref{EGJ-spread} with a deletion argument.

\begin{lemma}[First stage] \label{first-stage}  
  Fix $\eps \in (0, 1)$, and let $n$ be sufficiently large in terms of $\eps$. Then there exists a $(2/n, 4)$-spread random subset $S_1$ of $[n]^2$ that always satisfies the following properties:
  \begin{enumerate}
    \item[(a)] $|S_1| \ge (1-n^{-\Omega_{\eps}(1)}) n$;
    \item[(b)] $S_1$ contains at most one point on each $\eps$-heavy line;
    \item[(c)] for every line $L \in \mc{L}_{\eps}$ and bisector $\ell \in \mc{L}^{\perp}_{\eps}$, we have $|\{p \in S_1 : R_{\ell}(p) \in L\}| = O(n^{0.1})$;  
    \item[(d)] for every pair of distinct bisectors $\ell_1, \ell_2 \in \mc{L}^{\perp}_{\eps}$, we have 
    \[
    \big|\big\{(p_1, p_2) : p_1 \in S_1 \setminus \ell_1, \, p_2 \in S_1 \setminus \ell_2,\, R_{\ell_1}(p_1) = R_{\ell_2}(p_2)\big\}\big| = O(n^{0.1}).
    \]
  \end{enumerate}
\end{lemma}
\begin{proof}
  Let $\cH_{\eps}$ be the following $|\mc{D}_{\eps}|$-uniform $|\mc{D}_{\eps}|$-partite hypergraph: the vertices of $\cH_{\eps}$ are  the $\eps$-heavy lines intersecting $[n]^2$, and for each $p\in [n]^2$ we put an edge $\{L \in V(\cH_{\eps}) : p \in L\}$.
  Clearly, $\Delta(\cH_{\eps}) \le n$ and $\Delta_2(\cH_{\eps}) \le 1$.
  
  Let $w_{\size}$ be the $1$-uniform test function on $E(\cH_{\eps})$ such that $w_{\size}(p) = 1$ for each $p \in E(\cH_{\eps})$. 
  Clearly, $w_{\size}(E(\cH_{\eps})) = n^2$ and $B_n(w_{\size}) = n$.
  Also, for a line $L \in \mc{L}_{\eps}$ and a bisector $\ell \in \mc{L}^{\perp}_{\eps}$, let $w_{L, \ell}:E(\cH_{\eps}) \to \{0,1\}$ be the $1$-uniform test function such that $w_{L, \ell}(p) = 1$ if and only if $R_{\ell}(p) \in L$. Clearly, $w_{L, \ell}(E(\cH_{\eps})) \le n$ and $B_n(w_{L, \ell}) = n$.
  Furthermore, for each pair of distinct bisectors $\ell_1, \ell_2 \in \mc{L}^{\perp}_{\eps}$, let $w_{\ell_1, \ell_2}:\binom{E(\cH_{\eps})}{2} \to \{0,1\}$ be the $2$-uniform test function such that $w_{\ell_1, \ell_2}(\{p_1, p_2\}) = 1$ if and only if 
  \begin{itemize}
    \item $R_{\ell_1}(p_1) = R_{\ell_2}(p_2)$ or $R_{\ell_1}(p_2) = R_{\ell_2}(p_1)$, and 
    \item the line through $p_1$ and $p_2$ is $\eps$-light (this ensures that $w_{\ell_1, \ell_2}$ is clean).
  \end{itemize}
  Note that, for every $p \in [n]^2$ and $\ell_1, \ell_2 \in \mc{L}^{\perp}_{\eps}$, if $w_{\ell_1, \ell_2}(\{p, p'\}) = 1$ then $p' \in \{R_{\ell_1}(R_{\ell_2}(p)), R_{\ell_2}(R_{\ell_1}(p))\}$. Hence, $w_{\ell_1, \ell_2}(E(\cH_{\eps})) \le 2n^2$ and $B_n(w_{\ell_1, \ell_2}) \le \max(2n, n^2) = n^2$.
  
  Applying \cref{EGJ-spread} to $\cH_{\eps}$ with $\Delta = n$, $r = |\mc{D}_{\eps}| = O(1/\eps^2)$, $\delta = 0.1$, $L = 2$, and test functions $w_{\size}$, $(w_{L, \ell})_{L \in \mc{L}_{\eps}, \ell \in \mc{L}^{\perp}_{\eps}}$, and $(w_{\ell_1, \ell_2})_{\ell_1, \ell_2 \in \mc{L}^{\perp}_{\eps}}$, we obtain a $(2/n, 4)$-spread random matching $\mc{M}$ in $\cH_{\eps}$ that satisfies $w_{\size}(\mc{M}) = (1 \pm n^{-\Omega_{\eps}(1)})n$, and $w_{L, \ell}(\mc{M}) = O(n^{0.1})$ for each $\eps$-heavy line $L$ and bisector $\ell \in \mc{L}^{\perp}_{\eps}$, and $w_{\ell_1, \ell_2}(\mc{M}) = O(n^{0.1})$ for each pair of distinct bisectors $\ell_1, \ell_2 \in \mc{L}^{\perp}_{\eps}$. This matching corresponds to a $(2/n, 4)$-spread random subset $S_1$ of $[n]^2$, which satisfies (b) by the definition of $\cH_{\eps}$, and satisfies (a) and (c) by the above estimates on $w_{\size}(\mc{M})$ and $w_{L, \ell}(\mc{M})$. 
  Finally, for each pair of points $(p_1, p_2)$ counted in property (d), $p_1 \neq p_2$ and (by property (b)) the line through $p_1$ and $p_2$ is $\eps$-light. Thus, the number of such pairs of points is at most $2w_{\ell_1, \ell_2}(\mc{M}) = O(n^{0.1})$.
\end{proof}

\begin{lemma}[Second stage] \label{second-stage}  
  Fix $\eps \in (0, 1)$, and let $n$ be sufficiently large in terms of $\eps$. Also fix an arbitrary outcome of $S_1$ given by \cref{first-stage}. Then there exists a $(2/n, 4)$-spread random subset $S_2$ of $[n]^2 \setminus S_1$ that always satisfies the following properties:
  \begin{enumerate}
    \item[(a)] $|S_2| \ge (1-n^{-\Omega_{\eps}(1)}) n$;
    \item[(b)] $S_2$ contains at most one point on each $\eps$-heavy line;
    \item[(c)] $S_1 \cup S_2$ does not contain $\eps$-heavy isosceles trapezia. 
  \end{enumerate}
\end{lemma}
\begin{proof}
  Let $\cH'_{\eps}$ be the $2|\mc{D}_{\eps}|$-uniform $2|\mc{D}_{\eps}|$-partite hypergraph with ``line parts'' $(V_{\dir})_{\dir \in \mc{D}_{\eps}}$  and ``bisector parts'' $(V_{\dir}^{\perp})_{\dir \in \mc{D}_{\eps}}$ defined as follows:
  \begin{itemize}
    \item for each $\dir \in \mc{D}_{\eps}$, let $V_{\dir} \defeq \mc{L}_{\dir}$;
    \item for each $\dir \in \mc{D}_{\eps}$, let $V_{\dir}^{\perp} \defeq \mc{L}^{\perp}_{\dir} \cup \{v_{\dir, p} : p \in [n]^2\}$ (we refer to $v_{\dir, p}$ as a \emph{dummy vertex});
    \item the edges correspond to the points of $[n]^2 \setminus S_1$: namely, for each $\dir \in \mc{D}_{\eps}$, an edge corresponding to a point $p$ contains the vertices $L_{\dir}(p) \in V_{\dir}$ and $\ell^{\perp}(p, p_{\dir}) \in V_{\dir}^{\perp}$ where $p_{\dir}$ is the unique point of $S_1$ on the line $L_{\dir}(p)$ (if such a point $p_{\dir}$ does not exist then it contains the dummy vertex $v_{\dir, p} \in V_{\dir}^{\perp}$ instead).
  \end{itemize}

  First, we check that $\Delta(\cH'_{\eps}) \le n$. Indeed, for a vertex $L \in \mc{L}_{\dir}$, we have $\deg_{\cH'_{\eps}}(L) \le |L \cap [n]^2| \le n$. For a vertex $\ell \in \mc{L}^{\perp}_{\dir}$, we have $\deg_{\cH'_{\eps}}(\ell) \le |\{R_{\ell}(p) : p \in S_1\}| \le |S_1| \le n$. Each dummy vertex $v_{\dir, p}$ has degree at most one. 

  To bound the codegrees, consider two distinct vertices $v_1, v_2 \in V(\cH'_{\eps})$. If one of them is a dummy vertex, then $\deg_{\cH'_{\eps}}(v_1, v_2) \le 1$. If $v_1$ is a line $L_1 \in \mc{L}_{\dir_1}$ and $v_2$ is a line $L_2 \in \mc{L}_{\dir_2}$ then $\deg_{\cH'_{\eps}}(v_1, v_2) \le |L_1 \cap L_2| \le 1$. If $v_1$ is a line $L \in V_{\dir_1}$ and $v_2$ is a bisector $\ell \in \mc{L}^{\perp}_{\dir_2}$, then every edge $p$ containing both $v_1$ and $v_2$ satisfies $p \in L$ and $R_{\ell}(p) \in S_1$. Hence, the number of such edges is $O(n^{0.1})$ by \cref{first-stage}(c). 
  Finally, if $v_1$ and $v_2$ are two bisectors $\ell_1 \in \mc{L}^{\perp}_{\dir_1}$ and $\ell_2 \in \mc{L}^{\perp}_{\dir_2}$, then every edge $p$ containing both $v_1$ and $v_2$ satisfies $p \notin \ell_1 \cup \ell_2$, $R_{\ell_1}(p) \in S_1$, and $R_{\ell_2}(p) \in S_1$. Hence, the number of such edges is $O(n^{0.1})$ by \cref{first-stage}(d). In summary, we conclude that $\Delta_2(\cH'_{\eps}) = O(n^{0.1})$.

  As before, let $w_{\size}$ be the $1$-uniform test function on $E(\cH'_{\eps})$ such that $w_{\size}(p) = 1$ for each $p \in E(\cH'_{\eps})$. Applying \cref{EGJ-spread} to $\cH'_{\eps}$ with $\Delta = n$, $r = 2|\mc{D}_{\eps}| = O(1/\eps^2)$, $\delta = 0.1$, $L = 1$, and the test function $w_{\size}$, we obtain a $(2/n, 4)$-spread random matching $\mc{M}'$ in $\cH'_{\eps}$ such that $|\mc{M}'| = w_{\size}(\mc{M}') = (1 \pm n^{-\Omega_{\eps}(1)})n$. This matching corresponds to a $(2/n, 4)$-spread random subset $S_2$ of $[n]^2 \setminus S_1$ which satisfies (a) by the above size estimate, and satisfies (b) and (c) by the definition of $\cH'_{\eps}$.
\end{proof}

\begin{proof}[\textbf{Proof of \cref{no-four-on-a-circle-2n}}]
  Let $\eps \defeq c\eta$ for a sufficiently small absolute constant $c > 0$. Let $S_1$ be a random subset of $[n]^2$ given by \cref{first-stage}, and let $S_2$ be a random subset of $[n]^2 \setminus S_1$ given by \cref{second-stage}. Since $S_1$ is $(2/n, 4)$-spread, and $S_2$ is $(2/n, 4)$-spread conditionally on an arbitrary outcome of $S_1$, by \cref{spread-union}, $S_1 \cup S_2$ is $(4/n, 4)$-spread. 
  
  By \cref{first-stage}(a) and \cref{second-stage}(a), we have $|S_1 \cup S_2| \ge (2-n^{-\Omega_{\eps}(1)})n$. By \cref{first-stage}(b) and \cref{second-stage}(b), $S_1 \cup S_2$ contains at most 2 points on each $\eps$-heavy line. By \cref{second-stage}(c), $S_1 \cup S_2$ contains no $\eps$-heavy isosceles trapezia. Let $\mc{C}$ be the collection of sets $C \subseteq [n]^2$ of size $4$ such that $C$ is either contained in an $\eps$-light line, forms an $\eps$-light isosceles trapezium, or forms a cyclic quadrilateral that is not an isosceles trapezium. Let $S_{\del}$ be the subset of $S_1 \cup S_2$ obtained by including one point from each set $C \in \mc{C}$ contained in $S_1 \cup S_2$. By \cref{light-lines-triples,asymmetric-count,light-trapezia}, we have $|\mc{C}| = O(\eps n^5)$, and thus 
  \[
  \EE[|S_{\del}|] \le \sum_{C \in \mc{C}} \PP[C \subseteq S_1 \cup S_2] \le |\mc{C}| \cdot (4/n)^4 = O(\eps n).
  \]
  So, there is an outcome of $S_1$ and $S_2$ such that $|S_{\del}| = O(\eps n)$. In this case, the set $S \defeq (S_1 \cup S_2) \setminus S_{\del}$ satisfies 
  \[
  |S| \ge (2 - n^{-\Omega_{\eps}(1)}) n - O(\eps n) \ge (2 - \eta) n,
  \]
  by our choice of $\eps$, and does not contain four points on a circle or on a line by construction.
\end{proof}
\begin{remark}\label{rmk:third-stage}
The set $S = (S_1 \cup S_2) \setminus S_{\del}$ produced by our proof seems to be typically far from saturated, in the sense that there are still many points that can be added without violating the no-four-on-a-circle constraint. It is plausible that by tracking a lot of additional information in the first and second stages and using an inclusion-exclusion argument (similar to the one used in the proof of \cref{initial-configuration}), one could prove rigorous bounds along these lines, that would allow one to consider a third application of \cref{EGJ-spread}. This would provide roughly $\lambda n$ additional points (for some absolute constant $\lambda > 0$; back-of-the-envelope calculations suggest $\lambda \approx 0.19$) and imply that $f_{\cir}(n) \ge (2+\lambda-o(1))n$. 
\end{remark}

\bibliographystyle{plain}
\bibliography{main.bib}

@misc{conlon,
  author       = {David Conlon},
  howpublished = {Private communication},
  year         = {},
  month        = {},
  note         = {}
}

@inproceedings{dvir-lovett,
  title={Subspace evasive sets},
  author={Dvir, Zeev and Lovett, Shachar},
  booktitle={Proceedings of the forty-fourth annual ACM symposium on Theory of computing},
  pages={351--358},
  year={2012}
}

@inproceedings {jain-pham-24,
    AUTHOR = {Jain, Vishesh and Pham, Huy Tuan},
     TITLE = {Optimal thresholds for {L}atin squares, {S}teiner triple
              systems, and edge colorings},
 BOOKTITLE = {Proceedings of the 2024 {A}nnual {ACM}-{SIAM} {S}ymposium on
              {D}iscrete {A}lgorithms ({SODA})},
     PAGES = {1425--1436},
 PUBLISHER = {SIAM, Philadelphia, PA},
      YEAR = {2024},
      ISBN = {978-1-61197-791-2},
   MRCLASS = {68Q87},
  MRNUMBER = {4699302},
       DOI = {10.1137/1.9781611977912.57},
       URL = {https://doi.org/10.1137/1.9781611977912.57},
}

@article{alphaevolve,
      title={Alpha{E}volve: A coding agent for scientific and algorithmic discovery}, 
      author={Alexander Novikov and Ngân Vũ and Marvin Eisenberger and Emilien Dupont and Po-Sen Huang and Adam Zsolt Wagner and Sergey Shirobokov and Borislav Kozlovskii and Francisco J. R. Ruiz and Abbas Mehrabian and M. Pawan Kumar and Abigail See and Swarat Chaudhuri and George Holland and Alex Davies and Sebastian Nowozin and Pushmeet Kohli and Matej Balog},
      year={2025},
      eprint={2506.13131},
      archivePrefix={arXiv},
      primaryClass={cs.AI},
      url={https://arxiv.org/abs/2506.13131}, 
      note={Preprint, arXiv:2506.13131},
}

@article {BH05,
    AUTHOR = {Ball, S. and Hirschfeld, J. W. P.},
     TITLE = {Bounds on {$(n,r)$}-arcs and their application to linear
              codes},
   JOURNAL = {Finite Fields Appl.},
  FJOURNAL = {Finite Fields and their Applications},
    VOLUME = {11},
      YEAR = {2005},
    NUMBER = {3},
     PAGES = {326--336},
      ISSN = {1071-5797,1090-2465},
   MRCLASS = {51E22 (94B05)},
  MRNUMBER = {2158768},
MRREVIEWER = {Raymond\ Hill},
       DOI = {10.1016/j.ffa.2005.04.002},
       URL = {https://doi.org/10.1016/j.ffa.2005.04.002},
}

@book{dudeney-1917,
  author    = {Dudeney, H. E.},
  title     = {Amusements in Mathematics},
  year      = {1917},
  publisher = {Nelson},
  address   = {London}
}

@book{eppstein-18,
  author    = {Eppstein, David},
  title     = {Forbidden Configurations in Discrete Geometry},
  year      = {2018},
  publisher = {Cambridge University Press}
}

@incollection{brass-moser-pach-05,
  author    = {Brass, Peter and Moser, William O. J. and Pach, J{\'a}nos},
  title     = {Lattice Point Problems},
  booktitle = {Research Problems in Discrete Geometry},
  publisher = {Springer},
  address   = {New York},
  year      = {2005},
  pages     = {417--433},
  doi       = {10.1007/0-387-29929-7_11}
}

@unpublished{lefmann-12,
  author       = {Lefmann, Hanno},
  title        = {Extensions of the No-Three-In-Line Problem},
  year         = {2012},
  note         = {Preprint},
  url          = {https://www.tu-chemnitz.de/informatik/ThIS/downloads/publications/lefmann_no_three_submitted.pdf}
}

@article {HJSW-75,
    AUTHOR = {Hall, R. R. and Jackson, T. H. and Sudbery, A. and Wild, K.},
     TITLE = {Some advances in the no-three-in-line problem},
   JOURNAL = {J. Combinatorial Theory Ser. A},
  FJOURNAL = {Journal of Combinatorial Theory. Series A},
    VOLUME = {18},
      YEAR = {1975},
     PAGES = {336--341},
      ISSN = {0097-3165},
   MRCLASS = {10E99 (05B30)},
  MRNUMBER = {366817},
MRREVIEWER = {Richard\ K.\ Guy},
       DOI = {10.1016/0097-3165(75)90043-6},
       URL = {https://doi.org/10.1016/0097-3165(75)90043-6},
}

@article{guy-kelly-68,
  author  = {Guy, Richard K. and Kelly, Patrick A.},
  title   = {The No-Three-In-Line Problem},
  journal = {Canadian Mathematical Bulletin},
  volume  = {11},
  number  = {4},
  year    = {1968},
  pages   = {527--531},
  doi     = {10.4153/CMB-1968-062-3}
}

@misc {KNS-25,
      title={Settling the no-$(k+1)$-in-line problem when $k$ is not small}, 
      author={Benedek Kovács and Zoltán Lóránt Nagy and Dávid R. Szabó},
      year={2025},
      eprint={2502.00176},
      archivePrefix={arXiv},
      primaryClass={math.CO},
      url={https://arxiv.org/abs/2502.00176},
      note={Preprint, arXiv:2502.00176},
}

@misc {KNS-25-algebraic,
      title={Randomised algebraic constructions for the no-$(k+1)$-in-line problem}, 
      author={Benedek Kovács and Zoltán Lóránt Nagy and Dávid R. Szabó},
      year={2025},
      eprint={2508.07632},
      archivePrefix={arXiv},
      primaryClass={math.CO},
      url={https://arxiv.org/abs/2508.07632}, 
      note={Preprint, arXiv:2508.07632},
}

@unpublished{green-100-open-problems,
  author = {Green, Ben},
  title  = {100 Open Problems},
  note   = {Manuscript},
  url    = {https://people.maths.ox.ac.uk/greenbj/papers/open-problems.pdf},
  urldate= {2025-09-15}
}

@article {schmidt-68,
    AUTHOR = {Schmidt, Wolfgang M.},
     TITLE = {Asymptotic formulae for point lattices of bounded determinant
              and subspaces of bounded height},
   JOURNAL = {Duke Math. J.},
  FJOURNAL = {Duke Mathematical Journal},
    VOLUME = {35},
      YEAR = {1968},
     PAGES = {327--339},
      ISSN = {0012-7094,1547-7398},
   MRCLASS = {10.25},
  MRNUMBER = {224562},
MRREVIEWER = {E.\ S.\ Barnes},
       URL = {http://projecteuclid.org/euclid.dmj/1077377618},
}

@article{brass-knauer-03,
  author  = {Brass, Peter and Knauer, Christian},
  title   = {On counting point-hyperplane incidences},
  journal = {Comput. Geom.},
  volume  = {25},
  number  = {1--2},
  pages   = {13--20},
  year    = {2003},
  note    = {European Workshop on Computational Geometry (CG01)},
  doi     = {10.1016/S0925-7721(02)00127-X}
}

@article {sudakov-tomon-24,
    AUTHOR = {Sudakov, Benny and Tomon, Istv\'an},
     TITLE = {Evasive sets, covering by subspaces, and point-hyperplane
              incidences},
   JOURNAL = {Discrete Comput. Geom.},
  FJOURNAL = {Discrete \& Computational Geometry. An International Journal
              of Mathematics and Computer Science},
    VOLUME = {72},
      YEAR = {2024},
    NUMBER = {3},
     PAGES = {1333--1347},
      ISSN = {0179-5376,1432-0444},
   MRCLASS = {51E20 (52C10 94B05 94B27)},
  MRNUMBER = {4804978},
       DOI = {10.1007/s00454-023-00601-1},
       URL = {https://doi.org/10.1007/s00454-023-00601-1},
}

@misc {ghosal-goenka-keevash-25,
      title={On subsets of lattice cubes avoiding affine and spherical degeneracies}, 
      author={Anubhab Ghosal and Ritesh Goenka and Peter Keevash},
      year={2025},
      eprint={2509.06935},
      archivePrefix={arXiv},
      primaryClass={math.CO},
      url={https://arxiv.org/abs/2509.06935},
      note={To appear in Discrete Comput. Geom., arXiv:2509.06935}, 
}

@inproceedings {luria-simkin-22,
    AUTHOR = {Simkin, Michael and Luria, Zur},
     TITLE = {A lower bound for the {$n$}-queens problem},
 BOOKTITLE = {Proceedings of the 2022 {A}nnual {ACM}-{SIAM} {S}ymposium on
              {D}iscrete {A}lgorithms ({SODA})},
     PAGES = {2185--2197},
 PUBLISHER = {[Society for Industrial and Applied Mathematics (SIAM)],
              Philadelphia, PA},
      YEAR = {2022},
      ISBN = {978-1-61197-707-3},
   MRCLASS = {68W20},
  MRNUMBER = {4415123},
       DOI = {10.1137/1.9781611977073.86},
       URL = {https://doi.org/10.1137/1.9781611977073.86},
}

@article {simkin-23,
    AUTHOR = {Simkin, Michael},
     TITLE = {The number of {$n$}-queens configurations},
   JOURNAL = {Adv. Math.},
  FJOURNAL = {Advances in Mathematics},
    VOLUME = {427},
      YEAR = {2023},
     PAGES = {Paper No. 109127, 83},
      ISSN = {0001-8708,1090-2082},
   MRCLASS = {05A05 (05A16 05B30)},
  MRNUMBER = {4597952},
MRREVIEWER = {Eugenijus\ Manstavi\v cius},
       DOI = {10.1016/j.aim.2023.109127},
       URL = {https://doi.org/10.1016/j.aim.2023.109127},
}

@misc{grebennikov-kwan-25,
      title={No-$(k+1)$-in-line problem for large constant $k$}, 
      author={Alexandr Grebennikov and Matthew Kwan},
      year={2025},
      eprint={2510.17743},
      archivePrefix={arXiv},
      primaryClass={math.CO},
      url={https://arxiv.org/abs/2510.17743}, 
      note={Preprint, arXiv:2510.17743},
}

@article {ehard-glock-joos-20,
    AUTHOR = {Ehard, Stefan and Glock, Stefan and Joos, Felix},
     TITLE = {Pseudorandom hypergraph matchings},
   JOURNAL = {Combin. Probab. Comput.},
  FJOURNAL = {Combinatorics, Probability and Computing},
    VOLUME = {29},
      YEAR = {2020},
    NUMBER = {6},
     PAGES = {868--885},
      ISSN = {0963-5483,1469-2163},
   MRCLASS = {05C65 (05C15 05C70 05D15 05D40)},
  MRNUMBER = {4173135},
MRREVIEWER = {Ioan\ Tomescu},
       DOI = {10.1017/s0963548320000280},
       URL = {https://doi.org/10.1017/s0963548320000280},
}

@article{widmer-12,
  author   = {Widmer, Martin},
  title    = {Lipschitz class, narrow class, and counting lattice points},
  journal  = {Proceedings of the American Mathematical Society},
  volume   = {140},
  number   = {2},
  pages    = {677--689},
  year     = {2012},
  doi      = {10.1090/S0002-9939-2011-10926-2},
  mrnumber = {2846337}
}

@article {BMNR-20,
    AUTHOR = {B\'ar\'any, Imre and Martin, Greg and Naslund, Eric and
              Robins, Sinai},
     TITLE = {Primitive points in rational polygons},
   JOURNAL = {Canad. Math. Bull.},
  FJOURNAL = {Canadian Mathematical Bulletin. Bulletin Canadien de
              Math\'ematiques},
    VOLUME = {63},
      YEAR = {2020},
    NUMBER = {4},
     PAGES = {850--870},
      ISSN = {0008-4395,1496-4287},
   MRCLASS = {52C05 (52B20)},
  MRNUMBER = {4176774},
MRREVIEWER = {Mizan\ R.\ Khan},
       DOI = {10.4153/s0008439520000090},
       URL = {https://doi.org/10.4153/s0008439520000090},
}

@misc{bowtell-keevash-21,
      title={The $n$-queens problem}, 
      author={Candida Bowtell and Peter Keevash},
      year={2021},
      eprint={2109.08083},
      archivePrefix={arXiv},
      primaryClass={math.CO},
      url={https://arxiv.org/abs/2109.08083}, 
      note={Preprint, arXiv:2109.08083},
}

@misc{prellberg-26,
      title={Constraint Satisfaction Programming for the No-three-in-line Problem}, 
      author={Thomas Prellberg},
      year={2026},
      eprint={2602.07751},
      archivePrefix={arXiv},
      primaryClass={math.CO},
      url={https://arxiv.org/abs/2602.07751}, 
      note={Preprint, arXiv:2602.07751},
}

@article {roth-51,
    AUTHOR = {Roth, K. F.},
     TITLE = {On a problem of {H}eilbronn},
   JOURNAL = {J. London Math. Soc.},
  FJOURNAL = {The Journal of the London Mathematical Society},
    VOLUME = {26},
      YEAR = {1951},
     PAGES = {198--204},
      ISSN = {0024-6107},
   MRCLASS = {10.0X},
  MRNUMBER = {41889},
MRREVIEWER = {P. Scherk},
       DOI = {10.1112/jlms/s1-26.3.198},
       URL = {https://doi.org/10.1112/jlms/s1-26.3.198},
}

@article{DKP-26,
  author  = {Delcourt, Michelle and Kelly, Tom and Postle, Luke},
  title   = {Thresholds for {$(n,q,2)$}-{S}teiner systems via refined absorption},
  journal = {Mathematical Proceedings of the Cambridge Philosophical Society},
  year    = {2026},
  pages   = {1--20},
  doi     = {10.1017/S0305004126102059}
}

@article {SSS-23,
    AUTHOR = {Sah, Ashwin and Sawhney, Mehtaab and Simkin, Michael},
     TITLE = {Threshold for {S}teiner triple systems},
   JOURNAL = {Geom. Funct. Anal.},
  FJOURNAL = {Geometric and Functional Analysis},
    VOLUME = {33},
      YEAR = {2023},
    NUMBER = {4},
     PAGES = {1141--1172},
      ISSN = {1016-443X,1420-8970},
   MRCLASS = {05B07 (60C05)},
  MRNUMBER = {4616696},
MRREVIEWER = {Luc\ Teirlinck},
       DOI = {10.1007/s00039-023-00639-6},
       URL = {https://doi.org/10.1007/s00039-023-00639-6},
}

@misc{keevash-22,
      title={The optimal edge-colouring threshold}, 
      author={Peter Keevash},
      year={2022},
      eprint={2212.04397},
      archivePrefix={arXiv},
      primaryClass={math.CO},
      url={https://arxiv.org/abs/2212.04397}, 
      note={Preprint, arXiv:2212.04397},
}

@article{KKKMO-23,
    AUTHOR = {Kang, Dong Yeap and Kelly, Tom and K\"uhn, Daniela and
              Methuku, Abhishek and Osthus, Deryk},
     TITLE = {Thresholds for {L}atin squares and {S}teiner triple systems:
              bounds within a logarithmic factor},
   JOURNAL = {Trans. Amer. Math. Soc.},
  FJOURNAL = {Transactions of the American Mathematical Society},
    VOLUME = {376},
      YEAR = {2023},
    NUMBER = {9},
     PAGES = {6623--6662},
      ISSN = {0002-9947,1088-6850},
   MRCLASS = {05B15 (05B05 05B07 05C80)},
  MRNUMBER = {4630786},
MRREVIEWER = {Carl\ Johan\ Casselgren},
       DOI = {10.1090/tran/8954},
       URL = {https://doi.org/10.1090/tran/8954},
}

@misc{dong-xu-25,
      title={Large grid subsets without many cospherical points}, 
      author={Zichao Dong and Zijian Xu},
      year={2025},
      eprint={2506.18113},
      archivePrefix={arXiv},
      primaryClass={math.CO},
      url={https://arxiv.org/abs/2506.18113}, 
      note={Preprint, arXiv:2506.18113},
}

@book{guy-81,
  author    = {Richard K. Guy},
  title     = {Unsolved Problems in Number Theory},
  series    = {Unsolved Problems in Intuitive Mathematics},
  volume    = {1},
  publisher = {Springer-Verlag},
  address   = {New York},
  year      = {1981}
}

@phdthesis{thiele-thesis,
  author  = {Torsten Thiele},
  title   = {Geometric Selection Problems and Hypergraphs},
  school  = {Institut f{\"u}r Mathematik II, Freie Universit{\"a}t Berlin},
  address = {Berlin},
  year    = {1995},
  type    = {{PhD} thesis}
}

@article{thiele-95,
  author    = {Thiele, Torsten},
  title     = {The no-four-on-circle problem},
  journal   = {Journal of Combinatorial Theory, Series A},
  volume    = {71},
  number    = {2},
  pages     = {332--334},
  year      = {1995},
  doi       = {10.1016/0097-3165(95)90007-1},
  mrnumber  = {1342453}
}

@article{molloy-reed-00,
  author = {Molloy, Michael and Reed, Bruce},
  title = {Near-optimal list colorings},
  journal = {Random Structures \& Algorithms},
  volume = {17},
  number = {3--4},
  pages = {376--402},
  doi = {10.1002/1098-2418(200010/12)17:3/4<376::AID-RSA10>3.0.CO;2-0},
  url = {https://onlinelibrary.wiley.com/doi/abs/10.1002/1098-2418%28200010/12%2917%3A3/4%3C376%3A%3AAID-RSA10%3E3.0.CO%3B2-0},
  eprint = {https://onlinelibrary.wiley.com/doi/pdf/10.1002/1098-2418%28200010/12%2917%3A3/4%3C376%3A%3AAID-RSA10%3E3.0.CO%3B2-0},
  year = {2000}
}

@article {huxley-konyagin-09,
    AUTHOR = {Huxley, M. N. and Konyagin, S. V.},
     TITLE = {Cyclic polygons of integer points},
   JOURNAL = {Acta Arith.},
  FJOURNAL = {Acta Arithmetica},
    VOLUME = {138},
      YEAR = {2009},
    NUMBER = {2},
     PAGES = {109--136},
      ISSN = {0065-1036,1730-6264},
   MRCLASS = {11P21 (11E25)},
  MRNUMBER = {2520131},
MRREVIEWER = {Don\ Redmond},
       DOI = {10.4064/aa138-2-2},
       URL = {https://doi.org/10.4064/aa138-2-2},
}

@article {glock-joos-kim-kuhn-lichev-24,
    AUTHOR = {Glock, Stefan and Joos, Felix and Kim, Jaehoon and K\"uhn,
              Marcus and Lichev, Lyuben},
     TITLE = {Conflict-free hypergraph matchings},
   JOURNAL = {J. Lond. Math. Soc. (2)},
  FJOURNAL = {Journal of the London Mathematical Society. Second Series},
    VOLUME = {109},
      YEAR = {2024},
    NUMBER = {5},
     PAGES = {Paper No. e12899, 78},
      ISSN = {0024-6107,1469-7750},
   MRCLASS = {05C65 (05C70)},
  MRNUMBER = {4745877},
MRREVIEWER = {Yan\ Wang},
       DOI = {10.1112/jlms.12899},
       URL = {https://doi.org/10.1112/jlms.12899},
}

@article {alon-yuster-05,
    AUTHOR = {Alon, Noga and Yuster, Raphael},
     TITLE = {On a hypergraph matching problem},
   JOURNAL = {Graphs Combin.},
  FJOURNAL = {Graphs and Combinatorics},
    VOLUME = {21},
      YEAR = {2005},
    NUMBER = {4},
     PAGES = {377--384},
      ISSN = {0911-0119,1435-5914},
   MRCLASS = {05C70 (05C65)},
  MRNUMBER = {2209008},
MRREVIEWER = {Andr\'e\ E.\ K\'ezdy},
       DOI = {10.1007/s00373-005-0628-x},
       URL = {https://doi.org/10.1007/s00373-005-0628-x},
}

@article {pippenger-spencer-89,
    AUTHOR = {Pippenger, Nicholas and Spencer, Joel},
     TITLE = {Asymptotic behavior of the chromatic index for hypergraphs},
   JOURNAL = {J. Combin. Theory Ser. A},
  FJOURNAL = {Journal of Combinatorial Theory. Series A},
    VOLUME = {51},
      YEAR = {1989},
    NUMBER = {1},
     PAGES = {24--42},
      ISSN = {0097-3165,1096-0899},
   MRCLASS = {05C65 (05C15 05C70)},
  MRNUMBER = {993646},
       DOI = {10.1016/0097-3165(89)90074-5},
       URL = {https://doi.org/10.1016/0097-3165(89)90074-5},
}

@article {suk-zeng-26,
    AUTHOR = {Suk, Andrew and Zeng, Ji},
     TITLE = {On higher dimensional point sets in general position},
   JOURNAL = {Combin. Probab. Comput.},
  FJOURNAL = {Combinatorics, Probability and Computing},
    VOLUME = {35},
      YEAR = {2026},
    NUMBER = {1},
     PAGES = {134--148},
      ISSN = {0963-5483,1469-2163},
   MRCLASS = {52C35 (52C10)},
  MRNUMBER = {5021264},
       DOI = {10.1017/s0963548325100254},
       URL = {https://doi.org/10.1017/s0963548325100254},
}

\appendix\section{Deduction of \cref{EGJ-spread} from \cite{ehard-glock-joos-20}}
\label{sec:appendix}

First, we deduce the following version of \cref{EGJ-spread} that requires that $w(E(\cH))$ is large compared to $B_{\Delta}(w)$ for each test function $w$.

\begin{lemma} \label{EGJ-spread-restricted}
  Fix $\delta \in (0, 1)$ and $r, L \in \NN$ with $r \ge 2$. Let $\gamma \defeq \delta/(50L^2 r^2)$, and let $\Delta$ be sufficiently large in terms of $\delta, r, L$. Let $\cH$ be an $r$-uniform hypergraph satisfying $\Delta(\cH) \le \Delta$ and $\Delta_2(\cH) \le \Delta^{1-\delta}$ and $e(\cH) \le \exp(\Delta^{\gamma^2})$. Suppose that for each $\ell \in [L]$ we are given a set of clean $\ell$-uniform test functions $\mc{W}_{\ell}$ on $E(\cH)$ of size at most $\exp(\Delta^{\gamma^2})$ such that $w(E(\cH)) \ge B_{\Delta}(w) \Delta^{\delta}$ for all $w \in \mc{W}_{\ell}$. Then there exists a $((1 + \Delta^{-\gamma})/\Delta, \Delta^{\gamma})$-spread random matching $\mc{M}$ in $\cH$ that always satisfies
  \[
  w(\mc{M}) = \frac{(1 \pm \Delta^{-\gamma})w(E(\cH))}{\Delta^{\ell}}
  \]
  for each $\ell \in [L]$ and $w \in \mc{W}_{\ell}$.
\end{lemma}

\begin{proof}
  Set 
  \begin{equation} \label{eq:choice-of-P-Q-M}
  P \defeq \Delta^{20 L r \gamma}, \quad Q \defeq \Delta^{1 - 20(r-1+1/(4L))Lr\gamma}, \quad M \defeq \frac{(1 + 4\Delta^{-2\gamma}) \Delta}{P^{r-1} Q} = (1 + 4\Delta^{-2\gamma}) \Delta^{5r\gamma}.
  \end{equation}
  As in \cite{ehard-glock-joos-20}, we consider the following three-step randomised construction.
  \begin{itemize}
    \item Step 1: Consider a random partition $V(\cH) \defeq V_1 \sqcup \ldots \sqcup V_P$ obtained by assigning each vertex independently to a uniformly random part.
    \item Step 2: For each $i \in [P]$, let $\cH_i$ be a random subgraph of $\cH[V_i]$ obtained by including each edge independently with probability $1/Q$.
    \item Step 3: Using a theorem of Molloy and Reed \cite[Theorem 2]{molloy-reed-00}, for each $i \in [P]$ we partition the edges of $\cH_i$ into $M$ matchings $\mc{M}_{i, 1}, \ldots, \mc{M}_{i, M}$. Choosing $j_i \in [M]$ uniformly at random for each $i \in [P]$, we obtain the matching $\mc{M} \defeq \bigcup_{i \in [P]} \mc{M}_{i, j_i}$ on the entire vertex set.
  \end{itemize}
  The argument in \cite[Proof of Theorem 1.3]{ehard-glock-joos-20} shows that with probability at least $1 - \exp(-\Delta^{\gamma/2})$ this construction is well-defined (i.e., the hypergraphs $\cH_i$ satisfy the necessary conditions for the application of the Molloy--Reed theorem), and the resulting matching $\mc{M}$ satisfies 
  \[
  w(\mc{M}) = \frac{(1 \pm \Delta^{-\gamma})w(E(\cH))}{\Delta^{\ell}}
  \]
  for each $\ell \in [L]$ and $w \in \mc{W}_{\ell}$. Denote this event by $\mc{E}$. We will show that the distribution of $\mc{M}$ conditional on $\mc{E}$ is $((1 + \Delta^{-\gamma})/\Delta, \Delta^{\gamma})$-spread. 

  Fix a non-empty set $E = \{e_1, \ldots, e_N\} \subseteq E(\cH)$ of size $N \le \Delta^{\gamma}$. We may also assume that $E$ is a matching, since otherwise the probability of $E \subseteq \mc{M}$ is zero. Let $f_E:[N] \to [P]$ be such that $e_i \in E(\cH[V_{f_E(i)}])$ for each $i \in [N]$. For each function $f:[N] \to [P]$, we bound the probability of the event that $E \subseteq \mc{M}$ and $f_E = f$. For this event to occur, in step 1 we need to have $e_i \subseteq V_{f(i)}$ for each $i \in [N]$, which happens with probability $P^{-N r}$; in step 2 we need each edge $e_i$ to be included in $\cH_{f(i)}$, which happens with probability $Q^{-N}$; and in step 3 we need each edge $e_i$ to be included in the chosen matching $\mc{M}_{f(i), j_{f(i)}}$, which happens with probability at most $M^{-N_0}$ where $N_0 \defeq |f([N])|$. Taking the product, we obtain that $\PP[E \subseteq \mc{M} \text{ and } f_E = f] \le P^{-N r} Q^{-N} M^{-N_0}$, and by the union bound over all functions $f$ we have
  \begin{align*}
    \PP[E \subseteq \mc{M}] = \sum_{f:[N] \to [P]} \PP[E \subseteq \mc{M} \text{ and } f_E = f] &\le P^{-N r} Q^{-N} \sum_{N_0 = 1}^N \big|\big\{f: [N] \to [P] : |f([N])| = N_0\big\}\big| M^{-N_0} \\
    &\le P^{-N r} Q^{-N} \sum_{N_0 = 1}^N S(N, N_0) P^{N_0} M^{-N_0},
  \end{align*}
  where $S(N, N_0)$ is the number of partitions of an $N$-element set into $N_0$ non-empty parts (also known as Stirling numbers of the second kind). It is easy to check that $S(N, N_0) \le N^{2(N - N_0)}$, and thus
  \[
    \PP[E \subseteq \mc{M}] \le P^{-N r} Q^{-N} \sum_{N_0 = 1}^N N^{2(N - N_0)} P^{N_0} M^{-N_0} = (P^{r-1} Q M)^{-N} \sum_{N_1 = 0}^{N-1} (N^2 M/P)^{N_1}.
  \]
  Since $P^{r-1} Q M = (1 + 4\Delta^{-2\gamma}) \Delta$ and $N^2 M / P = O(\Delta^{-13Lr\gamma})$ by our choice of parameters \cref{eq:choice-of-P-Q-M}, we conclude that 
  \[
  \PP[E \subseteq \mc{M} \mid \mc{E}] \le \frac{\PP[E \subseteq \mc{M}]}{\PP[\mc{E}]} \le \frac{((1 + 4\Delta^{-2\gamma}) / \Delta)^N \cdot (1 + O(\Delta^{-13Lr\gamma}))}{1 - \exp(-\Delta^{\gamma/2})} \le ((1 + \Delta^{-\gamma})/\Delta)^N. \qedhere
  \]
\end{proof}

\cref{EGJ-spread} follows from \cref{EGJ-spread-restricted} by introducing \emph{phantom edges} to artificially increase $w(E(\mathcal{H}))$.

\begin{proof}[\textbf{Proof of \cref{EGJ-spread}}]
  Let $\cH^*$ be the hypergraph obtained from $\cH$ by adding a matching $F$ of $s \defeq \lceil 2L\Delta^{1+\delta}\rceil$ new $r$-edges, vertex-disjoint from $V(\cH)$. Then $\Delta(\cH^*) \le \Delta$ and $\Delta_2(\cH^*) \le \Delta^{1-\delta}$, and
  \[
  e(\cH^*) = e(\cH) + s \le \exp(\Delta^{(2\gamma)^2}).
  \]
  Fix $\ell \in [L]$ and $w \in \mc{W}_{\ell}$, and set
  \[
  T_{w} \defeq \max\{w(E(\cH)), B_{\Delta}(w)\Delta^{\delta}\}, \qquad
  R_{w} \defeq T_{w} - w(E(\cH)).
  \]
  Define the $\ell$-uniform test function $\hat{w}$ on $E(\cH^*)$ by setting
  \[
  \hat{w}(Y) \defeq
  \begin{cases}
    w(Y) & \text{if } Y \subseteq E(\cH),\\
    R_{w}/\binom{s}{\ell} & \text{if } Y \subseteq F,\\
    0 & \text{otherwise.}
  \end{cases}
  \]
  Clearly, $\hat{w}$ is clean and $\hat{w}(E(\cH^*)) = T_{w}$. Moreover, for $j \in [\ell]$, the $j$-degrees coming from $F$ are at most
  \[
  R_{w}\frac{\binom{s-j}{\ell-j}}{\binom{s}{\ell}} \le R_{w}\left(\ell/s\right)^j \le R_{w}\Delta^{-j-\delta}.
  \]
  It follows that
  \[
  B_{\Delta}(\hat{w}) \le \max\{B_{\Delta}(w), R_{w}\Delta^{-\delta}\} \le T_{w}\Delta^{-\delta},
  \]
  and thus $\hat{w}(E(\cH^*)) \ge B_{\Delta}(\hat{w})\Delta^{\delta}$.

  So, we can apply \cref{EGJ-spread-restricted} to $\cH^*$ and the families $\hat{\mc{W}}_{\ell} \defeq \{\hat{w}: w \in \mc{W}_{\ell}\}$ (with $2\gamma$ in place of $\gamma$) to obtain a $((1+\Delta^{-\gamma})/\Delta, \Delta^{\gamma})$-spread random matching $\mc{M}^*$ in $\cH^*$ which always satisfies
  \[
  \hat{w}(\mc{M}^*) = \frac{(1 \pm \Delta^{-\gamma})T_{w}}{\Delta^{\ell}}
  \]
  for every $\ell \in [L]$ and $w \in \mc{W}_{\ell}$. Then $\mc{M} \defeq \mc{M}^* \cap E(\cH)$ is a $((1+\Delta^{-\gamma})/\Delta,\Delta^{\gamma})$-spread random matching in $\cH$. It remains to verify that
  \begin{equation} \label{eq:desired-EGJ-condition} 
  w(\mc{M}) = \frac{(1 \pm \Delta^{-\gamma})w(E(\cH)) \pm 2 B_{\Delta}(w)\Delta^{\delta}}{\Delta^{\ell}}
  \end{equation}
  for every $\ell \in [L]$ and $w \in \mc{W}_{\ell}$. Indeed, if $R_{w} = 0$ then
  \[
  w(\mc{M}) = \hat{w}(\mc{M}^*) = \frac{(1 \pm \Delta^{-\gamma})w(E(\cH))}{\Delta^{\ell}},
  \]
  which implies \cref{eq:desired-EGJ-condition}. Otherwise, $T_{w} = B_{\Delta}(w)\Delta^{\delta}$, and hence
  \[
  \frac{w(E(\cH)) - 2B_{\Delta}(w) \Delta^{\delta}}{\Delta^{\ell}} \le 0 \le w(\mc{M}) \le \hat{w}(\mc{M}^*) \le \frac{(1+\Delta^{-\gamma})B_{\Delta}(w)\Delta^{\delta}}{\Delta^{\ell}} \le \frac{2B_{\Delta}(w)\Delta^{\delta}}{\Delta^{\ell}},
  \]
  which also implies \cref{eq:desired-EGJ-condition}.
\end{proof}

\section{Numerical data for the no-four-on-a-circle problem}\label{sec:alphaevolve}

Here we present some numerical data for the no-four-on-a-circle problem, obtained using \emph{AlphaEvolve}.

Let $f_{\mathrm{IT}}(n)$ be the maximum size of a subset of $[n]^2$ with no four points forming an isosceles trapezium and no four points on a line.
We saw in \cref{asymmetric-count} that for large $n$ almost all cyclic quadrilaterals in $[n]^2$ are isosceles trapezia, so it seems plausible that $f_{\mathrm{IT}}(n)-f_{\cir}(n)=o(n)$. Empirically, the convergence rate in \cref{asymmetric-count} seems to be very slow, so $f_{\mathrm{IT}}(n)$ might be more illuminating than $f_{\cir}(n)$ for small $n$.

Recall from the introduction that Thiele~\cite{thiele-thesis,thiele-95} proved the upper bound $f_{\cir}(n) \le \lfloor(5n-3)/2\rfloor$. Actually, it is not hard to see that $\lfloor(5n-3)/2\rfloor$ is the maximum size of a subset of $[n]^2$ containing no isosceles trapezium whose parallel sides are horizontal or vertical, so this quantity is also an upper bound on $f_{\mathrm{IT}}(n)$.

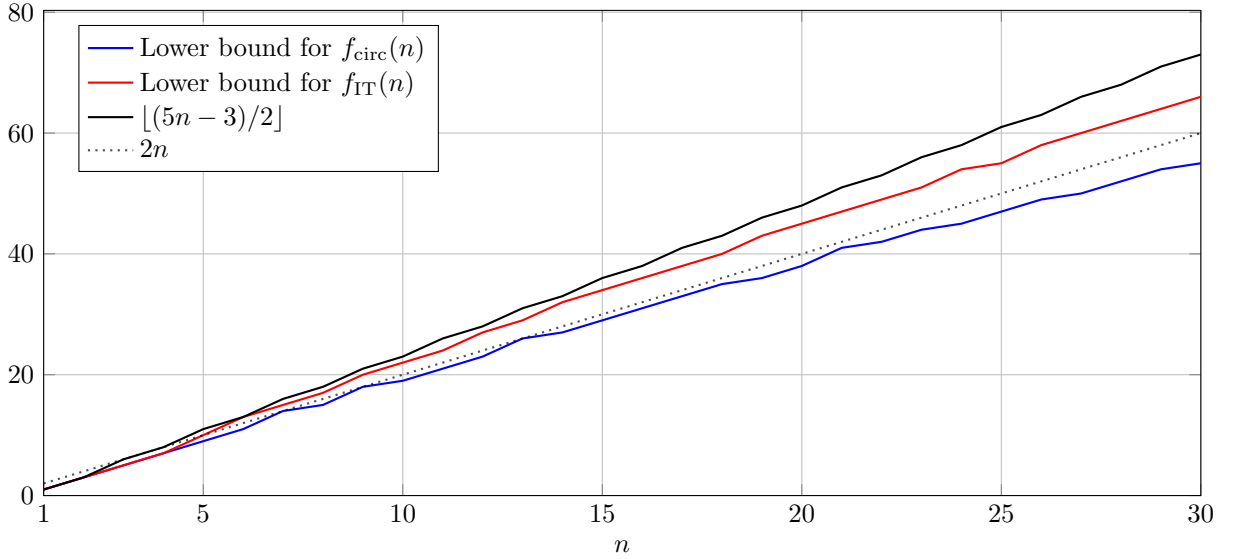
\pgfplotstableread[col sep=space]{
n circle noiso axisaligned
1	1	1	1
2	3	3	3
3	5	5	6
4	7	7	8
5	9	10	11
6	11	13	13
7	14	15	16
8	15	17	18
9	18	20	21
10	19	22	23
11	21	24	26
12	23	27	28
13	26	29	31
14	27	32	33
15	29	34	36
16	31	36	38
17	33	38	41
18	35	40	43
19	36	43	46
20	38	45	48
21	41	47	51
22	42	49	53
23	44	51	56
24	45	54	58
25	47	55	61
26	49	58	63
27	50	60	66
28	52	62	68
29	54	64	71
30	55	66	73
}\trapeziadata

\pgfplotstabletranspose[
    colnames from=n,
    input colnames to=quantity
]\trapeziatransposed{\trapeziadata}

\pgfplotstablegetrowsof{\trapeziadata}
\pgfmathtruncatemacro{\lastrow}{\pgfplotsretval-1}

\pgfplotstablegetelem{0}{n}\of{\trapeziadata}
\edef\firstn{\pgfplotsretval}

\pgfplotstablegetelem{\lastrow}{n}\of{\trapeziadata}
\edef\lastn{\pgfplotsretval}

\begin{table}[ht]
\centering
\scriptsize
\setlength{\tabcolsep}{3pt}
\resizebox{\textwidth}{!}{%
\pgfplotstabletypeset[
    string type,
    every head row/.style={
        before row=\toprule,
        after row=\midrule
    },
    every last row/.style={
        after row=\bottomrule
    },
    columns/quantity/.style={
        column name={$n$},
        string type,
        column type=l,
        postproc cell content/.code={%
            \ifnum\pgfplotstablerow=0
                \pgfkeyssetvalue{/pgfplots/table/@cell content}{$f_{\cir}(n)\ge $}%
            \fi
            \ifnum\pgfplotstablerow=1
                \pgfkeyssetvalue{/pgfplots/table/@cell content}{$f_{\mr{IT}}(n)\ge $}%
            \fi
            \ifnum\pgfplotstablerow=2
                \pgfkeyssetvalue{/pgfplots/table/@cell content}{$\lfloor(5n-3)/2\rfloor$}%
            \fi
        }
    }
]\trapeziatransposed%
}
\label{tab:trapezia-horizontal}
\end{table}

\begin{figure}[ht]
\centering
\begin{tikzpicture}
\begin{axis}[
    width=\textwidth,
    height=8cm,
    xlabel={$n$},
    ylabel={},
    xmin=1,
    xmax=30,
    ymin=0,
    grid=both,
    xtick={1,5,10,15,20,25,30},
    legend pos=north west,
    legend cell align=left,
]

\addplot+[thick, no markers,color=blue]
    table[x=n, y=circle] {\trapeziadata};
\addlegendentry{Lower bound for $f_{\cir}(n)$}

\addplot+[thick, no markers,color=red]
    table[x=n, y=noiso] {\trapeziadata};
\addlegendentry{Lower bound for $f_{\mathrm{IT}}(n)$}

\addplot+[thick, no markers,color=black]
    table[x=n, y=axisaligned] {\trapeziadata};
\addlegendentry{$\lfloor(5n-3)/2\rfloor$}

\addplot[
    thick,
    dotted,
    mark=none,
    color=black,
    draw opacity=0.7,
    domain=\firstn:\lastn,
    samples=2
]
    {2*x};
\addlegendentry{$2n$}

\end{axis}
\end{tikzpicture}
\label{fig:trapezia-chart}
\end{figure}
In an accompanying file with the arXiv version of the paper, we include the actual point sets certifying these lower bounds. We emphasise that there is no guarantee that these lower bounds are sharp.

\end{document}